\documentclass{siamltex}
\usepackage{graphicx}
\usepackage{bm}
\usepackage{amsmath}
\usepackage{amssymb}
\usepackage{epsfig}
\usepackage{psfig}
\usepackage{latexsym}
\usepackage{rotating}
\usepackage{amsbsy}
\usepackage{mathrsfs}
\usepackage{lineno}
\usepackage{stmaryrd}
\usepackage{xparse}
\usepackage[usenames, dvipsnames]{color}
\usepackage[breaklinks,colorlinks=true,allcolors=blue]{hyperref}

\pdfoutput=1    

\NewDocumentCommand{\dgal}{sO{}m}{%
  \IfBooleanTF{#1}
    {\dgalext{#3}}
    {\dgalx[#2]{#3}}%
}
\NewDocumentCommand{\dgalext}{m}{%
  \sbox0{%
    \mathsurround=0pt 
    $\left\{\vphantom{#1}\right.\kern-\nulldelimiterspace$%
  }%
  \sbox2{\{}%
  \ifdim\ht0=\ht2
    \{\kern-.45\wd2 \{#1\}\kern-.45\wd2 \}%
  \else
    \left\{\kern-.5\wd0\left\{#1\right\}\kern-.5\wd0\right\}%
  \fi
}

\NewDocumentCommand{\dgalx}{om}{%
  \sbox0{\mathsurround=0pt$#1\{$}%
  \sbox2{\{}%
  \ifdim\ht0=\ht2
    \{\kern-.45\wd2 \{#2\}\kern-.45\wd2 \}%
  \else
    \mathopen{#1\{\kern-.5\wd0 #1\{}
    #2
    \mathclose{#1\}\kern-.5\wd0 #1\}}
  \fi
}

\newtheorem{remark}{Remark}[section]

\newcommand{\sigmab}{\mbox{\boldmath$\sigma$}}

\font\msbm=msbm10

\newcommand{\R}{\hbox{{\msbm \char "52}}}

\definecolor{otherblue}{rgb}{0,0.3,0.6}
\def\rbl#1{{\textcolor{otherblue}{#1}}}

\title{{Robust a posteriori error estimators  for mixed approximation of nearly
incompressible elasticity}\thanks{{This work was supported 
by  EPSRC grant EP/P013317.}}}

\author{
Arbaz Khan\thanks{
School of Mathematics, University of Manchester, UK (\tt{arbaz.khan@manchester.ac.uk})}
\and
Catherine E. Powell\thanks{
School of Mathematics, University of Manchester, UK (\tt{c.powell@manchester.ac.uk})}
\and
David J. Silvester\thanks{
School of Mathematics, University of Manchester, UK (\tt{d.silvester@manchester.ac.uk})}.
}

\begin{document}

\maketitle

\begin{abstract}
This paper is concerned with the analysis and implementation of robust finite element approximation  methods for mixed formulations of linear elasticity problems where the elastic solid is almost incompressible. Several novel a posteriori error estimators for the energy norm of the finite element error are proposed and analysed. We establish upper and lower bounds for the energy error in terms of the proposed error estimators and prove that the constants in the bounds are independent of the Lam\'{e}  coefficients: thus the proposed estimators are robust in the incompressible limit.  Numerical  results are presented that validate the theoretical estimates. The software used to generate these results is available online.
\end{abstract}

\begin{keywords}
A posteriori analysis, {planar} elasticity,  finite elements, mixed approximation,  error estimation.
\end{keywords}

\begin{AMS} 65N30, 65N15.
\end{AMS}

\pagestyle{myheadings}

\thispagestyle{plain}
\markboth{A.\ KHAN, C.\  E.\ POWELL and D.\  J.\  SILVESTER}
{Error estimators for nearly incompressible elasticity}

\section{Introduction}\label{sec11}
The locking of finite element methods when solving nearly incompressible
elasticity problems is a  significant practical issue in computational engineering. 
The standard way of avoiding locking is to write the underlying equations as a system, by introducing an additional unknown (an auxiliary variable) which is related to pressure. We will adopt this strategy in this work, with the aim of developing  robust and effective a posteriori error estimation techniques for the  resulting mixed approximations.

{Our starting point is the classical linear elasticity problem}
\begin{subequations}  \label{os1}
\begin{align}  \label{os1a}
 -\nabla\cdot\sigmab& =\bm{f} \quad  \mbox{in } \Omega &&\hspace*{-24pt}(\mbox{{equilibrium} of forces}),\\
 \bm{u}&=\bm{g} \quad \mbox{on } \Gamma_D &&\hspace*{-24pt}(\mbox{{essential} boundary condition}),\\
 \sigmab \bm{n}&=\bm{0} \quad \mbox{on } \Gamma_N &&\hspace*{-24pt}(\mbox{{natural} boundary condition}),
  \label{os1c}
\end{align}
\end{subequations}
where $\Omega$ is a bounded Lipschitz polygon in $\R^2$ with {a} boundary 
$\Gamma = \partial\Omega= \Gamma_D \cup \Gamma_N$,
 {where} $\Gamma_D\cap\Gamma_N= \emptyset$. 
Here the linear elastic deformation of an isotropic {solid}  
 is written in terms of the stress tensor $\sigmab : \R^2\rightarrow \R^{2\times 2}$, the strain 
tensor $\bm{\varepsilon} : \R^2\rightarrow \R^{2\times 2}$, the body force 
$\bm{f}: \R^2\rightarrow \R^{2}$ and displacement field $\bm{u}: \R^2\rightarrow \R^{2} $.
The form of the stress tensor is given by
\begin{align*}
\sigmab=2 \mu \bm{\varepsilon}(\bm{u})+\lambda ({\nabla\cdot \bm{u}}){\bm{I}},
\end{align*}
where ${\bm{I}}$ is the $2\times 2$ identity matrix and we have  
\begin{align*}
\bm{\varepsilon}(\bm{u})=\frac{1}{2}(\nabla \bm{u}+(\nabla \bm{u})^{\top}).
\end{align*} 
Here $\mu$ and $\lambda$ are the Lam\'{e}  coefficients{: with $0<\mu_1<\mu<\mu_2 <\infty$} 
and $0<\lambda<\infty$. They can be written in 
terms of the Young's modulus $E$ and the Poisson ratio $\nu$ as 
\begin{align*}
\mu=\frac{E}{2(1+\nu)}, \quad \lambda=\frac{E\nu}{(1+\nu)(1-2\nu)}.
\end{align*}

{The mathematical issue that  underlies the phenomenon of locking is that
the coefficient $\lambda$ (and hence the stress tensor $\sigmab$) is unbounded  in the 
incompressible limit $\nu = 1/2$.}
{In this work we will consider {\it pressure robust}  approximation methods, which arise
from considering the following {\sl Herrmann} mixed  formulation \cite{RLH} of (\ref{os1a})--(\ref{os1c}),}
\begin{subequations} \label{os2a}
\begin{align}
 -\nabla\cdot\sigmab& =\bm{f} \quad\mbox{in } \Omega, \\
 \nabla\cdot\bm{u}+\frac{p}{\lambda} &=0\quad\mbox{in } \Omega,\\
 \bm{u}&=\bm{g}\quad\mbox{on } \Gamma_D,\\
 \sigmab \bm{n}&= {\bm{0}} \quad\mbox{on }\Gamma_N. \label{os2b}
\end{align}
\end{subequations}
Here, we have introduced the {auxiliary} variable  $p$  {(the so-called Herrmann pressure) and the stress tensor definition is given by}
\begin{align}
\sigmab=2 \mu \bm{\varepsilon}(\bm{u})- p{\bm{I}}.
\end{align}
There is an extensive literature on finite element approximation of elasticity problems; see
Boffi et al~\cite{DFM} for a comprehensive overview  and Hughes~\cite{TJH} for an 
engineering perspective. {Our belief is  that this paper is the first  comprehensive study 
of a posteriori error estimation techniques for mixed approximations of  planar
elasticity}.\footnote{Our analysis can easily be generalised
 to cover three-dimensional {isotropic} linear elasticity.}{An important feature  is that 
 the  constructed error estimators are {\it robust} in the sense that material parameters
do not appear in the norm-equivalence constants.} 
  {We note that other aspects} of mixed approximation of the {Herrmann} formulation 
have been {discussed} by  Stenberg and collaborators~\cite{KS,DR} and  
by Houston et al.~\cite{PDT} previously.  \rbl{Other relevant papers 
include~\cite{barrios2006residual,lonsing2004posteriori,becker2009nitsche,arnold2002mixed,arnold2007mixed}.}

The rest of the paper is organised as follows. {Section~\ref{Hdivmethsec} 
discusses finite element approximation of the Hermann formulation (\ref{os2a}).}
{A  detailed residual-based a posteriori error analysis is presented in Section~\ref{rposest}. Building on this, a selection of novel and  potentially more efficient local error problem estimators are discussed in Section~\ref{SLEPE}. Numerical results that complement the theory are then presented in the final section.}
\section{Approximation aspects}
\label{Hdivmethsec}
{Our notation is conventional:} $H^s(\omega)$ denotes {the}
usual  Sobolev space  with the associated norm
$||\cdot||_{s,\omega}$ for $s\ge0$. In {the case}
$\omega=\Omega$, we use $||\cdot||_{s}$ instead of
$||\cdot||_{s,\Omega}$. {We will denote
vector-valued Sobolev spaces} by boldface letters
$\bm{H}^{s}(\omega)=\bm{H}^{s}(\omega;\R^2)$. {We also define}
 \begin{align*}
\bm{H}^1_E(\Omega):=\bigl\{\bm{v}\in \bm{H}^1(\Omega)
                    \;\big|\; \bm{v}|_{\Gamma_D}=\bm{g}\bigr\}, \quad
\bm{H}^{1\over 2}(\Gamma_D):=\bigl\{\bm{v} \, |\, \bm{v}=\bm{u}|_{\Gamma_D}, \bm{u}\in \bm{H}^1(\Omega)\bigr\} ,
\end{align*} 
{and the test spaces} 
$$
 \bm{H}^1_{E_0}(\Omega) :=\bigl\{\bm{v}\in \bm{H}^1(\Omega)
                    \;\big|\; {\bm{v}|_{\Gamma_D}}=\bm{0}\bigr\} , \quad M :=  L^{2}(\Omega).
$$
{The standard weak formulation of (\ref{os2a})  is given by}: 
 find $(\bm{u},p)\in \bm{H}^1_E\times M$ such that 
 \begin{subequations} \label{scm11a}
\begin{align}
a(\bm{u},\bm{v})+b(\bm{v},p)&=f(\bm{v})\quad \forall \bm{v}\in\bm{H}^1_{E_0}, \\
b(\bm{u},q)-c(p,q)&=0\quad\quad\,\forall  q\in  M, \label{scm11b}
\end{align}
\end{subequations}
with forms defined so that
\begin{gather*}
a(\bm{u},\bm{v})=2\mu\int_{\Omega}\bm{\varepsilon}(\bm{u}):\bm{\varepsilon}(\bm{v}),
\quad  b(\bm{v},p)=-\int_{\Omega} p \nabla\cdot \bm{v}, \\
c(p,q)=\frac{1}{\lambda}\int_{\Omega} pq,
\quad f(\bm{v})=\int_{\Omega}\bm{f}\,\bm{v}.
\end{gather*}
{We will assume that the load function} $\bm{f}\in ( L^{2}(\Omega))^2$.
For convenience, the boundary data $\bm{g}\in \bm{H}^{1\over 2}(\Gamma_D)$ will be taken to be a polynomial of degree at most two {in each component}---this will ensure that no error is incurred in approximating the {essential} boundary condition on $\Gamma_D$. 
{Following convention, we also} define the bilinear form
\begin{align}
\mathcal{B}(\bm{u},p; \bm{v},q)=a(\bm{u},\bm{v})+b(\bm{v},p)+b(\bm{u},q)-c(p,q),
\end{align}
{so as to} express the formulation (\ref{scm11a}) in the compact form:
find $(\bm{u},p)\in \bm{H}^1_E\times  M$ such that 
\begin{align}\label{scm12}
\mathcal{B}(\bm{u},p; \bm{v},q)=f(\bm{v}), \quad \forall (\bm{v},q)\in\bm{H}^1_{E_0}\times  M.
\end{align}
The well-posedness of the formulation (\ref{scm12}) is addressed in the next two remarks.
\begin{remark}  
For a compressible material with $\nu\in\left (0,\frac{1}{2}\right )$, the existence and uniqueness 
of  {a weak solution satisfying}  (\ref{scm12}) is directly implied by the  
coercivity of $\mathcal{B}(\bm{u},p; \bm{u},-p)$ on $\bm{H}^1_{E_0}$, {see (\ref{infsup1})}. 
\end{remark}
\begin{remark}  
For $\nu=\frac{1}{2}$, the existence and uniqueness of  {a weak solution satisfying} 
(\ref{scm12}) is implied by the  
coercivity of $a(\bm{u},\bm{u})$ over $\bm{H}^1_{E_0}$ (by Korn's inequality) together  
with an {inf-sup condition satisfied by
$b(\bm{v},p)$ on $\bm{H}^1_{E_0}\times M$;  see \cite{QLDS}, \cite{VGPAR}, or \cite{DFM} for the proof.}
\end{remark}

To define the finite element approximation, we let $\{\mathcal{T}_{h}\}$ denote a family of shape regular rectangular meshes of $\bar{\Omega}$ into rectangles $K$ of diameter $h_K$. For  each $\mathcal{T}_h$, we define $\mathcal{E}_h$ as the set of all edges of 
$\mathcal{T}_h$ and $h_E$ as the length of the edge $E\in\mathcal{E}_h$.
To obtain the discrete weak formulation of (\ref{os2a}) {we introduce finite-dimensional 
subsets
$\bm{X}^h_E \subset \bm{H}^1_{E}$, $\bm{X}^h_0 \subset \bm{H}^1_{E_0}$ 
and $M^h \subset M$}. 
The discrete weak formulation {is then given by}:
 find $(\bm{u}_h,p_h)\in \bm{X}^h_E\times M^h$ such that 
 \begin{subequations} \label{FEA11}
\begin{align}
a(\bm{u}_h,\bm{v}_h)+b(\bm{v}_h,p_h)&=f(\bm{v}_h)\quad \forall  \bm{v}_h\in \bm{X}^h_0, \\
b(\bm{u}_h,q_h)-c(p_h,q_h)&=0\quad\quad\;\;\forall  q_h\in M^h.
\end{align}
\end{subequations}
{Analogous to (\ref{scm12}), the discrete formulation can also be written as}: 
find $(\bm{u}_h,p_h)\in \bm{X}^h_E\times M^h$ such that 
\begin{align}\label{FEA12}
\mathcal{B}(\bm{u}_h,p_h; \bm{v}_h,q_h)=f(\bm{v}_h) \quad \forall (\bm{v}_h,q_h)\in \bm{X}^h_0\times M^h.
\end{align}
Well-posedness of the discrete formulation is (essentially) immediate.
\begin{remark}  
For a compressible material with $\nu\in\left (0,\frac{1}{2}\right )$, the existence and uniqueness 
of  a discrete solution satisfying  (\ref{FEA11})  or  (\ref{FEA12}) is directly implied by the (inherited) coercivity of $\mathcal{B}(\bm{u}_h,p_h; \bm{u}_h,-p_h)$ on $\bm{X}^h_{0}$. 
\end{remark}
\begin{remark}  
For $\nu=\frac{1}{2}$, the existence and uniqueness of  a discrete solution satisfying
 (\ref{FEA11})  or  (\ref{FEA12}) is implied by the  (inherited)
coercivity of $a(\bm{u}_h,\bm{u}_h)$ over $\bm{X}^h_{0}$, 
together  with a {\it discrete} inf-sup condition 
 satisfied by $b(\bm{v}_h,p_h)$ on $\bm{X}^h_{0}\times M^h$.   This inf-sup condition is 
 associated  with the construction of stable {\it Stokes} elements in 
 incompressible flow modelling and is not automatic---it needs to be verified for specific choices of the approximation spaces $\bm{X}^h_{0}$ and $M^h$ on a case-by-case basis.
\end{remark}

{Note that the solution space} $\bm{X}^h_E$ is obtained from the space $\bm{X}^h_0$ {by construction:}
\begin{align*}
\bm{X}^{h}_E= \left \{ \bm{u}\Big|\bm{u}=\sum_{j=1}^{n_u} a_j\bm{\phi}_j
+\sum_{j=n_u+1}^{n_u+n_\partial} a_j\bm{\phi}_j \right \}
\end{align*}
with coefficients $a_j\in \mathbb{R}$ and associated {vector-valued} basis functions $\{\bm{\phi}_j\}_{j=1}^{n_u}$ 
that span $\bm{X}^h_0$. The additional coefficients $ \{a_j\}_{ j=n_u+1}^{n_u+n_\partial}$ are associated with 
 Lagrange interpolation of the boundary data $\bm{g}$ on $\Gamma_D$. 
The finite dimensional spaces $\bm{X}^h_0$ and $M^h$ are related to $\{\mathcal{T}_{h}\}$. 
While the analysis in the next section is applicable to \emph{any} conforming approximation pair, the focus in the final two sections of the paper is on the Taylor--Hood approximation pair  $\bm{Q}_2$--$\bm{Q}_1$, (which combines continuous biquadratic approximation of the components of the displacement with a continuous bilinear approximation of the pressure field) and the $\bm{Q}_2$--$\bm{P}_{-1}$ pair (which uses a discontinuous linear approximation of the pressure field). A key point is that both methods are  known to be inf--sup stable Stokes approximations in  two (and also in three) spatial dimensions; see Elman et al.~\cite{HDA} for a detailed discussion.

\section{Residual-based a posteriori error analysis}\label{rposest}

{The  error  analysis will be developed in the  (energy) norm:}
\begin{align}\label{enorm}
|||(\bm{u},p)|||^2&=2\mu\, {|| \nabla \bm{u}||^2_{0}}  +(2\mu)^{-1}||p||^2_{0}+\lambda^{-1}||p||^2_{0}.
\end{align}
Note that there is a natural extension  of (\ref{enorm}) to
the {\it Hydrostatic formulation} of linear elasticity discussed by Boffi \& Stenberg in~\cite{DR}. 
{Thus, mixed approximation  of the Hydrostatic formulation using  $\bm{Q}_2$--$\bm{Q}_1$
 and  $\bm{Q}_2$--$\bm{P}_{-1}$  is also covered by our analysis.}\footnote{{It is 
 straightforward to verify that both of these  mixed approximation methods
satisfy the additional {coercivity} condition that is discussed  in~\cite{DR}.}} 

\subsection{A residual error estimator}
First, we define some important parameters for the analysis, which have an explicit dependence on the local {grid size}, as well as the Lam\'e coefficients:
\begin{align}\label{gridparams}
\rho_K=h_K(2\mu)^{-\frac{1}{2}}/2,\quad \rho_E=h_E(2\mu)^{-1}/2,\quad\rho_d= 1/( \lambda^{-1} + (2\mu)^{-1}).
\end{align}
Next, we define a local error indicator $\eta_{K}$ for each {element} $K\in\mathcal{T}_{h}$. The square of this local error indicator is the sum of terms,
$\eta^2_{K}=\eta^2_{R_K}+\eta^2_{E_K}+\eta_{J_K}^2$, with 
\begin{align} \label{components} 
\eta^2_{R_K}=\rho_{K}^2||\bm{R}_K||^2_{0,K}, \quad \eta^2_{J_K}=\rho_d||R_K||^2_{0,K}, \quad 
\eta^2_{E_K}=\sum_{E\in {\partial K}}\rho_E||\bm{R}_E||^2_{0,E},
\end{align}
where the two {\it element} residuals are given by 
\begin{align}
\bm{R}_K=\left\{\bm{f}_{h}+\nabla\cdot(2\mu \bm{\varepsilon}(\bm{u}_{h}))-\nabla p_{h}\right\} \big|_K,
\quad R_K=\left \{\nabla\cdot \bm{u}_h+\frac{1}{\lambda}p_h \right\} \Big|_K,
\end{align}
and {the  {\it edge} residual is associated with  the normal stress jump, so that}
\begin{align}\label{stressjump_def}
\bm{R}_E=\left\{\begin{array}{ll}
\frac{1}{2}\llbracket(p_{h}{\bm{I}}-2\mu\bm{\varepsilon}(\bm{u}_{h}))
\bm{n}\rrbracket_E & E\in \mathcal{E}_h\setminus\Gamma ,\\
((p_{h}{\bm{I}}-2\mu\bm{\varepsilon}(\bm{u}_{h}))\bm{n})_E& E\in \mathcal{E}_h\cap\Gamma_N ,\\
0 & E\in \mathcal{E}_h\cap\Gamma_D .
\end{array}\right. 
\end{align}
We let $\bm{f}_h$ be a piecewise polynomial approximation of $\bm{f}$ {that is
possibly discontinuous across element edges and we associate it with the data oscillation term}
\begin{align}
 \label{dataapp1}
 \Theta_{K}^2=\rho_{K}^2||\bm{f}-\bm{f}_{h}||^{2}_{0,K}.
 \end{align}
The residual error estimator and data oscillation error are then defined respectively, by summing the element contributions to give
 \begin{align}\label{errest1}
 \eta=\left(\sum_{K\in\mathcal{T}_{h}}\eta_{K}^2  \right)^{1/2} \quad \hbox{and} \quad
  \Theta=\left(\sum_{K\in\mathcal{T}_{h}}\Theta_{K}^2\right)^{1/2}.
 \end{align}
 
 The estimator $\eta$ is a reliable and efficient energy norm error estimator for any conforming mixed  approximation satisfying (\ref{FEA11}).  Proofs of the following  theorems are presented in subsequent sections.  The symbols $\lesssim$ and $\gtrsim$ will be used to denote bounds that are valid up to positive constants---these will be independent of the local mesh parameters ($h_E$ and $h_K$) as well as the  Lam\'{e}  coefficients ($\mu$ and $\lambda$)  that are specified in the formulation of the elasticity problem. The first result is that the estimator $\eta$ in (\ref{errest1}) gives rise to a reliable a posteriori error bound.
 \begin{theorem}\label{realiab}
 Suppose that  $(\bm{u},p)$ is the weak solution
 satisfying (\ref{scm11a}) and that 
 $(\bm{u}_h,p_h)\in\bm{X}^h_E\times M^h$  is {a conforming} mixed approximation 
 satisfying (\ref{FEA11}).   Defining  $\eta$ and $\Theta$ to be the
 error estimator and the data oscillation term in (\ref{errest1}), we
have an upper bound on the approximation error,
   \begin{align}
     |||(\bm{u}-\bm{u}_h, p-p_h)|||\lesssim \eta +\Theta.
   \end{align}   
 \end{theorem}
The second theorem identifies a lower bound on the error and shows
 the efficiency of the error estimator.
 \begin{theorem}\label{efficie}
 Suppose that  $(\bm{u},p)$ is the weak solution
 satisfying (\ref{scm11a}) and that 
 $(\bm{u}_h,p_h)\in\bm{X}^h_E\times M^h$  is {a conforming}  mixed approximation 
 satisfying (\ref{FEA11}). Defining  $\eta$ and $\Theta$ to be the
 error estimator and the data oscillation term in (\ref{errest1}), we have a lower bound on the approximation error,
      \begin{align}\label{elowerbd}
     \eta\lesssim\, |||(\bm{u}-\bm{u}_h, p-p_h)||| +\Theta.
   \end{align}   
 \end{theorem}
 
\subsection{Preliminary results}\label{proofreeft}

In this section, we establish a few technical results that are needed for the proofs of Theorems~\ref{realiab} and~\ref{efficie}. First, we recall the following well known estimates:
 \begin{gather}
 a(\bm{v},\bm{v})\ge C_K 2\mu \, { ||\nabla \bm{v} ||_0^2} \quad \forall \bm{v}\in {\bm{H}^1_{E_0}},  
 \label{aell} \\
\inf_{0\neq {q}\in {M}} \sup_{0\neq\bm{v}\in {\bm{H}^1_{E_0}}}
\frac{b(\bm{v},q)}{ ||\nabla\bm{v}||_0  ||q||_{0}}\ge C_{\Omega},
\label{binfsup}\\
 a(\bm{u},\bm{v})\le  {2\mu} \, { ||\nabla \bm{u} ||_0}  \,
  \, { ||\nabla \bm{v} ||_0} \quad \forall \bm{u},\bm{v}\in
 {\bm{H}^1_{E_0}}. \label{abd}
 \end{gather}
 The first estimate is  a direct consequence of  Korn's inequality and is discussed by Brenner \cite{SB} and Brenner \& Sung~\cite{SCL}. The second can be found in Girault \& Raviart~\cite{VGPAR}, and the third follows directly from the Cauchy--Schwarz inequality (using the definition of the Frobenius norm combined with  Young's inequality for products). The stability of the weak formulation (\ref{scm12}), independent of the Lam\'e coefficients, can now readily be established as a consequence of these estimates.
\begin{lemma}\label{Sinsuplem12}
For any $(\bm{u},p)\in \bm{H}^1_{E_0}\times {M}$, there exists a pair of functions $(\bm{v},q)\in \bm{H}^1_{E_0}\times {M}$, with  $|||(\bm{v},q)|||\lesssim |||(\bm{u},p)|||$, satisfying
$$ \mathcal{B}(\bm{u},p; \bm{v},q)\gtrsim |||(\bm{u},p)|||^{{2}}. $$
\end{lemma}
\begin{proof} 
{First, since $p\in M=L^2(\Omega)$, a consequence of the continuous inf-sup condition (\ref{binfsup})
is that there exists a function} $\bm{v}\in {\bm{H}^1_{E_0}}$ satisfying
\begin{align*}
({p},\nabla\cdot\bm{v})\ge C_{\Omega} (2\mu)^{-1}||p||^2_{0},
\quad (2\mu)^{1/2}||\nabla\bm{v}||_0\le (2\mu)^{-1/2}||p||_{0},
\end{align*}
where $C_\Omega>0$ is {the  inf-sup constant.}  
 {Since $ \bm{u}\in \bm{H}^1_{E_0}$, (\ref{abd}) implies  that}
\begin{align}\label{infsup2}
\mathcal{B}(\bm{u},p; {-\bm{v}},0) &\ge C_{\Omega}(2\mu)^{-1}||p||^2_{0}
- {(2\mu)^{1/2}} { ||\nabla \bm{u} ||_0} \;(2\mu)^{1/2} { ||\nabla \bm{v} ||_0},\nonumber\\
&\ge C_{\Omega}(2\mu)^{-1}||p||^2_{0}- {(2\mu)^{1/2}} { ||\nabla \bm{u} ||_0} \;(2\mu)^{-1/2}||p||_0, \nonumber\\
&\ge \Bigg(C_{\Omega}-\frac{1}{\epsilon}\Bigg)(2\mu)^{-1}||p||^2_{0}- {\epsilon(2\mu)} { ||\nabla \bm{u} ||^2_0},
\end{align} 
{for al}l $\epsilon> 0$.
{Second, the coercivity estimate (\ref{aell}) gives the following  bound}
\begin{align}\label{infsup1}
\mathcal{B}(\bm{u},p;\bm{u},-p)\ge C_{K} 2\mu \, {||\nabla \bm{u} ||^2_0} +\frac{1}{\lambda}||p||^2_0,
\end{align}
where $C_K$ is {the}  Korn constant.

{Next, introducing  a parameter $\delta$ and combining  (\ref{infsup2}) and (\ref{infsup1}) gives}
\begin{align*}
\mathcal{B}(\bm{u},p;\bm{u} {- \delta \bm{v}},-p)
&=\mathcal{B}(\bm{u},p;\bm{u},-p)+\delta\mathcal{B}(\bm{u},p; {-\bm{v}},0)\\
&\ge C_K 2\mu  \, { ||\nabla \bm{u} ||^2_0} + \! \frac{1}{\lambda}||p||^2_0
+\!\delta \Bigg(\!C_{\Omega}\!-\!\frac{1}{\epsilon}\Bigg)(2\mu)^{-1}||p||^2_{0} \!-\! {\delta\epsilon}\, (2\mu) { ||\nabla \bm{u} ||^2_0}\\
&\ge (C_K- {\delta\epsilon}) 2\mu \, { ||\nabla \bm{u} ||^2_0}
+\Bigg(\frac{1}{\lambda}+\delta \Bigg(C_{\Omega}-\frac{1}{\epsilon}\Bigg)(2\mu)^{-1}\Bigg)||p||^2_{0}.
\end{align*}
{Making specific choices of parameters}
 $\epsilon =2/C_\Omega$, $\delta=C_\Omega C_K/{4}$, {leads to the required estimate
 with $\bm{v}:=  \bm{u} -\delta  \bm{v}$ and $q:=-p$,}
\begin{align}\label{infsup11}
\mathcal{B}(\bm{u},p;\bm{u}{-\delta\bm{v}},-p)
&\ge \min\Bigg\{1, \frac{C_K}{2},{\frac{ C_K C_\Omega^2}{8}} \Bigg\}   |||(\bm{u},p)|||^2.
\end{align}
{To complete the proof,  we note that}
\begin{align*}
{2\mu\,  {|| \nabla  \bm{u} -\delta \nabla \bm{v}} ||^2_{0}}  &\leq
2 \cdot 2\mu \,  {|| \nabla  \bm{u}  ||^2_{0}  +
{2 \delta^2} \cdot 2\mu \,  {|| \nabla \bm{v}} ||^2_{0}} \nonumber\\
&\leq
{2 (2\mu)}\,  {|| \nabla  \bm{u}  ||^2_{0}} +
{{2 \delta^2} \cdot (2\mu)^{-1} }\,  {|| p ||^2_{0}} ,
\end{align*}
{which leads to the upper bound, }
\begin{align}\label{infsup21}
|||(\bm{u} {-\delta\bm{v}},-p)|||^2  &= 
{2\mu\,  {|| \nabla  \bm{u} -\delta \nabla \bm{v}} ||^2_{0}}  +(2\mu)^{-1}||p||^2_{0}+\lambda^{-1}||p||^2_{0} \nonumber\\
&\le { \Bigg(2+\frac{C_K^2 C_\Omega^2}{8} \Bigg) |||(\bm{u},p)|||^2}.
\end{align}
The constants  in  (\ref{infsup11}) and  (\ref{infsup21}) are independent of  the Lam\'e coefficients.
\end{proof}

\smallskip 

\begin{lemma}[Cl\'{e}ment interpolation estimate]\label{approxlem11}
Given $\bm{v}\in {{\bm{H}^1_{E_0}}}$, let $\bm{v}_h\in \bm{X}^h_0 $ be the quasi-interpolant of $\bm{v}$ 
{defined by averaging as discussed in  Cl\'{e}ment~\cite{CLA}}. For any $K\in\mathcal{T}_h$ {we have}
 \begin{align*}
 \rho^{-1}_K||\bm{v}-\bm{v}_h||_{0,K}&\lesssim (2\mu)^{1/2} |\bm{v}|_{1,\omega_K},
 \end{align*}
 {where  $|\cdot |_{1,\omega_K}$ is the  $H^1(\omega_K)$ seminorm. Moreover,}  for all $E\in\partial K$  {we have}
 \begin{align*}
  \rho^{-1/2}_E||\bm{v}-\bm{v}_h||_{0,E}&\lesssim (2\mu)^{1/2} |\bm{v}|_{1,\omega_K},
  \end{align*}
 where $\omega_K$ is the set of rectangles sharing at least one vertex with $K$.
 \end{lemma}
 
 {\it Proof.}
 The first quasi-interpolation estimate is well known,
\begin{align*}
||\bm{v}-\bm{v}_h||_{0,K}&\lesssim h_K |\bm{v}|_{1,\omega_K},
\end{align*}
so that, using (\ref{gridparams}), we get
\begin{align*}
\rho^{-1}_K||\bm{v}-\bm{v}_h||_{0,K}&\lesssim \rho^{-1}_K h_K |\bm{v}|_{1,\omega_K}\lesssim (2\mu)^{1/2} |\bm{v}|_{1,\omega_K}.
\end{align*}
The second quasi-interpolation estimate is also well known,
\begin{align*}
||\bm{v}-\bm{v}_h||_{0,E}&\lesssim h_E^{1/2} |\bm{v}|_{1,\omega_K},
\end{align*}
leading to the desired estimate
\begin{align*}
\rho^{-1/2}_E||\bm{v}-\bm{v}_h||_{0,E}&\lesssim \rho^{-1/2}_E h_E^{1/2}
|\bm{v}|_{1,\omega_K} \lesssim (2\mu)^{1/2} |\bm{v}|_{1,\omega_K}.\quad \endproof
\end{align*}

\noindent 
Using the above results we can now prove Theorems \ref{realiab} and \ref{efficie}.

 \subsection{Proof of Theorem \ref{realiab}}
 From Lemma \ref{Sinsuplem12}, we have
\begin{align*}
||||(\bm{u}-\bm{u}_h,p-p_h )|||^{2} \lesssim \mathcal{B}(\bm{u}-\bm{u}_h,p-p_h;\bm{v},q)
\end{align*}
with $|||(\bm{v},q)|||\le { |||(\bm{u} -\bm{u}_h,{p-p_h})||}| $.
Using (\ref{scm12}) and (\ref{FEA12}) {gives}
\begin{align}\label{rea11}
\mathcal{B}(\bm{u}-\bm{u}_h,p-p_h;\bm{v},q)&=\mathcal{B}(\bm{u}-\bm{u}_h,p-p_h;\bm{v}-\bm{v}_h,q),\nonumber\\
&=(\bm{f},\bm{v}-\bm{v}_h)-2\mu(\bm{\varepsilon}(\bm{u}_h),\bm{\varepsilon}(\bm{v}-\bm{v}_h))
     +  (p_h, \nabla\cdot (\bm{v}-\bm{v}_h))  \nonumber\\
& \quad { - (q, \nabla\cdot \bm{u}) + (q, \nabla\cdot \bm{u}_h) 
  - \frac{1}{\lambda}(q, p) + \frac{1}{\lambda}(q, p_h)},\nonumber\\
&=(\bm{f} - \bm{f}_h,\bm{v}-\bm{v}_h) \nonumber\\
& \quad
 + \! \sum_{K\in \mathcal{T}_h}\!
\Big\{ \big(\nabla \cdot (2\mu\bm{\varepsilon}(\bm{u}_h))
-  \nabla p_h + \bm{f}_h,(\bm{v} - \bm{v}_h) \big)_{{0,K}} \nonumber\\
&\quad {+} \! \sum_{E\in \partial K} \!\big \langle \bm{R}_E, \bm{v}-\bm{v}_h \big\rangle_E
 + \big(q, \nabla\cdot \bm{u}_h+\frac{1}{\lambda}p_h\big)_{{0,K}} \Big\} 
\end{align}
where $ {\big\langle} \bm{R}_E, \bm{v}-\bm{v}_h {\big\rangle}_E=\int_E \bm{R}_E\cdot (\bm{v}-\bm{v}_h) $. 
{Applying Cauchy--Schwarz to (\ref{rea11}) gives}
\begin{align}\label{rea12}
& \mathcal{B}(\bm{u}-\bm{u}_h,p-p_h;\bm{v},q)
\nonumber \\&\qquad\leq 
C \; \Bigg\{ \Bigg(\sum_{K\in \mathcal{T}_h}\rho_K^2||\bm{f}-\bm{f}_h||_{0,K}^2\Bigg)^{1/2}
\Bigg(\sum_{K\in \mathcal{T}_h}\rho_K^{-2}||\bm{v}-\bm{v}_h||_{0,K}^2\Bigg)^{1/2}
\nonumber \\
&\qquad\quad
+  \Bigg(\sum_{K\in \mathcal{T}_h}\rho_K^2|| \bm{R}_K ||_{0,K}^2\Bigg)^{1/2}
 \Bigg(\sum_{K\in \mathcal{T}_h}\rho_K^{-2}||\bm{v}-\bm{v}_h||_{0,K}^2\Bigg)^{1/2}\nonumber\\
&\qquad\quad+\Bigg(\sum_{K\in \mathcal{T}_h}\sum_{E\in \partial K}\rho_E||\bm{R}_E||_{0,E}^2\Bigg)^{1/2}
\Bigg(\sum_{K\in \mathcal{T}_h}\sum_{E\in \partial K}\rho_E^{-1}||\bm{v}-\bm{v}_h||_{0,E}^2\Bigg)^{1/2}\nonumber\\
&\qquad\quad+\Bigg(\sum_{K\in \mathcal{T}_h}\rho_d|| R_K||_{0,K}^2\Bigg)^{1/2}
\Bigg(\sum_{K\in \mathcal{T}_h}\rho^{-1}_d||q||_{0,K}^2\Bigg)^{1/2} \Bigg\}. 
\end{align}
{Using Lemma \ref{approxlem11}, then leads to the desired upper bound}
\begin{align}\label{rea13}
{ ||||(\bm{u}-\bm{u}_h,p-p_h )|||^{{2}}} & \lesssim \mathcal{B}(\bm{u}-\bm{u}_h,p-p_h;\bm{v},q) \nonumber\\
& \lesssim \; \Bigg\{ \Bigg(\sum_{K\in \mathcal{T}_h} {\Theta_K^2} \Bigg)^{1/2}
\Bigg(\sum_{K\in \mathcal{T}_h}\rho_K^{-2}||\bm{v}-\bm{v}_h||_{0,K}^2\Bigg)^{1/2}
\nonumber \\
&\qquad\quad
+  \Bigg(\sum_{K\in \mathcal{T}_h}{\eta^2_{ R_K} } \Bigg)^{1/2}
 \Bigg(\sum_{K\in \mathcal{T}_h}\rho_K^{-2}||\bm{v}-\bm{v}_h||_{0,K}^2\Bigg)^{1/2}\nonumber\\
&\qquad\quad+\Bigg(\sum_{K\in \mathcal{T}_h}  {\eta^2_{E_K} }  \Bigg)^{1/2}
\Bigg(\sum_{K\in \mathcal{T}_h}\sum_{E\in \partial K}\rho_E^{-1}||\bm{v}-\bm{v}_h||_{0,E}^2\Bigg)^{1/2}\nonumber\\
&\qquad\quad+\Bigg(\sum_{K\in \mathcal{T}_h} {\eta^2_{ J_K} } \Bigg)^{1/2}
\Bigg(\sum_{K\in \mathcal{T}_h}\rho^{-1}_d||q||_{0,K}^2\Bigg)^{1/2} \Bigg\}. \nonumber \\
&\lesssim \Bigg(\sum_{K\in \mathcal{T}_h}\!\!\Big\{2\mu \; |\bm{v}|_{1,K}^2
\!+\! \rho_d^{-1}||q||_{0,K}^2 \Big\}\Bigg)^{1\over 2}
\Bigg(\!\sum_{K\in \mathcal{T}_h} \!\!\Big(\eta_{K}^2+\Theta_K^2\Big)\!\Bigg)^{1\over 2}  \nonumber\\
& \lesssim  \; {  |||(\bm{v},q )||| \;
\Bigg(\sum_{K\in \mathcal{T}_h}\Big(\eta_{K}^2+\Theta_K^2\Big)\Bigg)^{1\over 2} }.
\end{align}
\endproof
\subsection{Proof of Theorem \ref{efficie}}
To establish the lower bound (\ref{elowerbd}), we now need to establish efficiency bounds for each of the component residual terms $\eta_{ R_K}$, $\eta_{ J_K}$ and $\eta_{E_K}$ defined in \eqref{components}. The method of proof is well known; see, for example~\cite{QLDS}.

Let $K$ be {an} element of $\mathcal{T}_h$ {and suppose that  $\chi_K$ is 
a (quartic)  interior bubble  function (positive in the interior of $K$, zero on $\partial K$).  
Then the  following estimates hold,  see Verf\"urth~\cite{RV}.}
\begin{align}
||\chi_K \bm{v}||_{0,K}&\lesssim ||\bm{v}||_{0,K}\lesssim ||\chi_K^{1/2}\bm{v}||_{0,K},  \label{bubblea}\\
||\nabla(\chi_K\bm{v})||_{0,K}&\lesssim h_K^{-1}||\bm{v}||_{0,K}, \label{bubbleb}
\end{align}
where $ \bm{v}$ denotes a vector-valued polynomial function {defined} on $K$. 
\begin{lemma}\label{efficie12}
Let $K$ be an element of $\mathcal{T}_h$. The local equilibrium residual satisfies
\begin{align*}
\eta^2_{R_K}&\lesssim \Big(2\mu \; |\bm{u}-\bm{u}_h|_{1,K}^2+ (2\mu)^{-1}||p-p_h||_{0,K}^2+\Theta_K^2\Big).
\end{align*}
\end{lemma}
\begin{proof}
For each element $K$ in $\mathcal{T}_{h}$, we have
$ \bm{R}_K=(\bm{f}_{h}+\nabla\cdot(2\mu \bm{\varepsilon}(\bm{u}_{h}))-\nabla p_{h})|_K$. 
{Next, introducing  $\bm{w}|_K=\rho^{2}_K \bm{R}_K\, \chi_K $ and using (\ref{bubblea}) we have}
\begin{align*}
\eta^2_{R_K} =\rho^2_K ||\bm{R}_K||^2_{0,K} 
 \lesssim (\bm{R}_K,\rho^2_K \chi_K \bm{R}_K)_{K} 
=(\bm{f}_{h}+\nabla\cdot(2\mu \bm{\varepsilon}(\bm{u}_{h}))-\nabla p_{h},\bm{w})_{K}.
\end{align*}
{Noting that}
 $(\bm{f}+\nabla\cdot(2\mu \bm{\varepsilon}(\bm{u}))-\nabla p)|_K=0$ for {the exact solution} $(\bm{u},p)$, we simply subtract, then integrate by parts and note that $\bm{w}|_{\partial K}=\bm{0}$,  to give
\begin{align}\label{resef2}
\eta^2_{R_K}&= (2\mu\bm{\varepsilon}( \bm{u}-\bm{u}_{h}),\bm{\varepsilon}( \bm{w}))_{K}
+(p_{h}-p,\nabla \cdot \bm{w})_{K}+((\bm{f}_{h}-\bm{f}),\bm{w})_{K}.
\end{align}
{Applying  Cauchy--Schwarz to (\ref{resef2})  leads to the bound} 
\begin{align*}
\eta^2_{R_K}\lesssim &\Big(2\mu \: |  \bm{u}-\bm{u}_h|_{1,K}^2+ (2\mu)^{-1}||p-p_h||_{0,K}^2+\Theta_K^2\Big)^{1\over 2}
\Big(2\mu \: |\bm{w}|^{2}_{1,K}+\rho^{-2}_K||\bm{w}||^2_{0,K}\Big)^{1\over 2}.
\end{align*}
{Using  (\ref{bubblea}) and (\ref{bubbleb}) then gives the bound}
\begin{align*}
\eta^2_{R_K}&\lesssim \Big(2\mu \: |\bm{u}-\bm{u}_h|_{1,K}^2+ (2\mu)^{-1}||p-p_h||_{0,K}^2
+\Theta_K^2\Big)^{1/2}\Big(\eta^2_{R_K}\Big)^{1/2}
\end{align*}
{as required.}
\end{proof}
\begin{lemma}\label{efficieR2}
Let $K$ be an element of $\mathcal{T}_h$. The local mass conservation residual satisfies
\begin{align*}
\eta^2_{J_K}&\lesssim \Big(2\mu \; |\bm{u}-\bm{u}_h|_{1,K}^2+ (2\mu)^{-1}||p-p_h||_{0,K}^2
+ \lambda^{-1} ||p-p_h||_{0,K}^2\Big).
\end{align*}
\end{lemma}
\begin{proof}
Noting that $(\nabla\cdot \bm{u}+\frac{1}{\lambda} p)|_K=0$ for a classical solution $(\bm{u},p)$, we have
\begin{align}\label{resefd5}
\rho_d||\nabla\cdot \bm{u}_{h}+\frac{1}{\lambda}p_h||_{0,K}^2&=\rho_d||\nabla\cdot (\bm{u}-\bm{u}_{h})+\frac{1}{\lambda}(p-p_h)||_{0,K}^2\\
&\lesssim \rho_d||\nabla\cdot (\bm{u} -\bm{u}_h ) \, ||_{0,K}^2+\frac{\rho_d}{\lambda^2}||( p-p_h )||_{0,K}^2\\
&\lesssim 2\mu \; |\bm{u} -\bm{u}_h |_{1,K}^2+ \frac{1}{\lambda}||(p-p_h)||_{0,K}^2,
\end{align}
{where the last line follows from the definition  of $\rho_d$ in \eqref{gridparams}.}
\end{proof}

Next, let $E$ denote an interior edge which is shared by two elements $K$ and $K^\prime$,
and suppose that $\chi_E$ is a polynomial bubble function on $E$
 (positive in the interior of  the patch $\omega_E$ formed by the union of $K$ and ${K}^\prime$ and zero on the 
 boundary of the patch).  The following estimates are well known; see Verf\"urth~\cite{RV},
\begin{align}\label{efficie13}
||\bm{v}||_{0,E}&\lesssim ||\chi_E^{1/2}\bm{v}||_{0,E}\\
||{\chi_E \bm{v}} ||_{0,K}&\lesssim h_E^{1/2}||\bm{v}||_{0,E} \quad\forall K\in\omega_{E}, \label{efficie14}\\
||{\nabla(\chi_E \bm{v})} ||_{0,K}&\lesssim h_E^{-1/2}||\bm{v}||_{0,E}\quad \forall K\in\omega_{E}.\label{efficie15}
\end{align}
{Here,} $\bm{v}$ is a vector-valued polynomial function {defined}  on $E$, and
 $\bm{v}_\chi={\chi_E \bm{v}}$ can be extended by zero outside of the patch.
\begin{lemma}\label{efficieED1}
Let $K$ be an element of $\mathcal{T}_h$. The stress jump residual satisfies
\begin{align*}
{\eta_{E_K}^2} \lesssim {\sum_{E\in \partial K} \left(2\mu\, |\bm{u}-\bm{u}_h|_{1,\omega_E}^2
+(2\mu)^{-1}||p-p_h||_{0,\omega_E}^2+\Theta_{\omega_E}^2 \right)},
\end{align*}
where $\Theta_{\omega_E}^2 = \sum_{K\in \omega_E}\Theta_K^2$ {is the localised data oscillation term}.
\end{lemma}
\begin{proof} For each $K$, we have $\eta_{E_K}^2 = \sum_{E\in \partial K} \rho_E||\bm{R}_E||^2_{0,E}$.  Suppose $E$ is an interior edge and recall that the classical solution $(\bm{u},p)$ satisfies $ \llbracket (p{\bm{I}}-2\mu\bm{\varepsilon}(\bm{u}))\bm{n}\rrbracket|_E=0$. Now let $\chi_E$ be a polynomial bubble function associated with $E$ as above, and define the localised jump term
$ \Lambda=\rho_E \,
\llbracket (p_{h}{\bm{I}}-2\mu\bm{\varepsilon}(\bm{u}_{h}))\bm{n}\rrbracket \, \chi_E$. Using (\ref{efficie13}) and \eqref{stressjump_def} gives
\begin{align*}
\rho_E||\bm{R}_E||^2_{0,E}
&\lesssim
\big(\llbracket (p_{h}{\bm{I}}-2\mu\bm{\varepsilon}(\bm{u}_{h}))\bm{n}\rrbracket,\Lambda\big)_E \\
&= \big(\llbracket (p_{h}{\bm{I}}-2\mu\bm{\varepsilon}( \bm{u}_{h}))\bm{n}\rrbracket-\llbracket (p{\bm{I}}
-2\mu\bm{\varepsilon}( {\bm{u}} ))\bm{n}\rrbracket,\Lambda\big)_E.
\end{align*}
Integrating by parts over each  element in the patch $\omega_E$ then gives
\begin{align*}
\big(\llbracket (p_{h}{\bm{I}}-2\mu\bm{\varepsilon}(\bm{u}_{h}))\bm{n}\rrbracket
& -\llbracket (p{\bm{I}}-2\mu\bm{\varepsilon}(\bm{u}))\bm{n}\rrbracket,\Lambda \big)_E\\
&= \sum_{K\in{\omega}_E} \! \Big\{\int_{K} \!
\big\{- \! \nabla\cdot 2\mu (\bm{\varepsilon}(\bm{u})-\bm{\varepsilon}(\bm{u}_{h}) ) 
+\nabla (p-p_{h}) \big\} \, \cdot \,  \Lambda \\ 
&\quad+\int_{{K}} \! \big\{\! -2\mu (\bm{\varepsilon}(\bm{u})-\bm{\varepsilon}(\bm{u}_{h}))
+ (p-p_{h}) \, {\bm{I}} \big \} :\nabla\Lambda \Big\}   . 
\end{align*}
Now, since the { exact solution} $(\bm{u},p)$ satisfies $-\nabla \cdot 2 \mu\bm{\varepsilon}(\bm{u}) + \nabla p{|_K} = \bm{f}{|_K}$, we have
\begin{align*}
\rho_E||\bm{R}_E||^2_{0,E}
&\lesssim \sum_{K\in\omega_E}\int_{K}  \big\{ \bm{f}_{h}
+\nabla\cdot(2\mu\bm{\varepsilon}(\bm{u}_{h}))-\nabla p_{h} \big\} \, \cdot\, \Lambda  \\ 
&\quad +\sum_{K\in\omega_E}\int_{K}(\bm{f}-\bm{f}_{h})\cdot \Lambda \\ 
&\quad +\sum_{K\in\omega_E}\int_{K} \big\{ \! -2\mu (\bm{\varepsilon}(\bm{u})-\bm{\varepsilon}(\bm{u}_{h}))
+ (p-p_{h}) \, {\bm{I}} \big\} : \nabla\Lambda \\
&\lesssim T_1+T_2+T_3.
\end{align*} 
These three terms will be bounded separately. 

First, using the definition of $\bm{R}_{K}$, and then combining the Cauchy--Schwarz inequality with Lemma \ref{efficie12}, gives
\begin{align*}
T_1\lesssim \Big( 2\mu \; | \bm{u}-\bm{u}_h|_{1,\omega_E}^2+(2\mu)^{-1}||p-p_h||_{0,\omega_E}^2
+ {\Theta_{\omega_E}^2} \Big)^{1/2}
\left(\sum_{K\in\omega_E} \rho_{K}^{-2} ||\Lambda||^{2}_{0,K}\right)^{1/2}.
\end{align*}  
Next, given  the shape regularity of the grid, using the definition of $\Lambda$ and (\ref{efficie14}) gives
$$ \rho_{K}^{-2} ||\Lambda||^{2}_{0,K}  \lesssim  \rho_{E}^{-1} h_E^{-1}  ||\Lambda||^{2}_{0,K} \lesssim  \rho_{E}^{-1} || \, \rho_{E}  \, \bm{R}_E \,  ||^2_{0,E} .$$
Hence, the following estimate holds
 \begin{align*}
T_1\lesssim \Big( 2\mu|\bm{u}-\bm{u}_h|_{1,\omega_E}^2+(2\mu)^{-1}||p-p_h||_{0,\omega_E}^2
+ {\Theta_{\omega_E}^2} \Big)^{1/2}
{\left(\sum_{K\in\omega_E}  \! \rho_{E}    ||  \bm{R}_E \,  ||^2_{0,E}\right)^{1/2}}.
\end{align*}
{Second, combining the Cauchy--Schwarz inequality with the above construction gives}
\begin{align*}
T_2 \lesssim     \Big(\sum_{K\in\omega_E}\rho_K^2  || \bm{f}-\bm{f}_{h} ||^2_{0,K} \Big)^{1/2}
\Big(\sum_{K\in\omega_E} \rho_{K}^{-2} ||\Lambda||^{2}_{0,K}\Big)^{1/2} 
\!\! \lesssim {\Theta_{\omega_E}}
 {\Big(\!\sum_{K\in\omega_E}  \! \rho_{E}    ||  \bm{R}_E \,  ||^2_{0,E}\Big)^{1/2}}\!\! \!\!.
\end{align*} 
{The third term can be bounded in a similar way, }
\begin{align*}
T_3 &\lesssim 
\Big( 2\mu|\bm{u}-\bm{u}_h|_{1,\omega_E}^2+(2\mu)^{-1}||p-p_h||_{0,\omega_E}^2 \Big)^{1/2} 
{\Big(\sum_{K\in\omega_E} 2\mu \, || \nabla \Lambda||^{2}_{0,K}\Big)^{1/2}},
\end{align*} 
where this time the second term is bounded using (\ref{efficie15}),
 $$ 2\mu \, || \nabla \Lambda||^{2}_{0,K} 
  \lesssim  \rho_{E}^{-1} h_E   || \nabla \Lambda||^{2}_{0,K} 
\lesssim    \rho_{E}^{-1}  || \, \rho_{E}  \, \bm{R}_E \,  ||^2_{0,E} .
$$

Combining the upper bounds for  $T_1, T_2$ and $T_3$, for any interior edge $E$, we have
\begin{align*}
\rho_E ||\bm{R}_E||_{0,E}^{2}  \lesssim \Big( 2\mu|\bm{u}-\bm{u}_h|_{1,\omega_E}^2+(2\mu)^{-1}||p-p_h||_{0,\omega_E}^2  + \Theta_{\omega_E}^{2} \Big). 
\end{align*}
If $E \in \Gamma_{N}$, then the same result holds with $\omega_{E}=K$ and we recall from \eqref{stressjump_def} that $\bm{R}_{E}=0$ for edges $E$ on $\Gamma_{D}$. Hence, summing over all the edges of element $K$, gives the required result.
\end{proof}

Finally, the lower bound (\ref{elowerbd}) in Theorem~\ref{efficie} follows from consolidating the estimates from Lemma~\ref{efficie12}, Lemma~\ref{efficieR2} and Lemma~\ref{efficieED1}.

\section{{Local problem error} estimators}\label{SLEPE}   
Having established that the residual error estimator $\eta$ in (\ref{errest1}) is reliable and efficient, the framework established by  Verf\"urth~\cite{RV}, makes it straightforward to now construct equivalent {\it local problem} estimators that are equally reliable and potentially more efficient.  We will
discuss four { novel robust local error estimators } in this section---the actual performance of two of these options (the so-called Stokes and Poisson problem local error estimators) will be discussed in Section~\ref{Numres}. To keep the discussion concise, the proofs of the equivalence of the estimators will simply be sketched. 

\subsection{{Elasticity problem local} error estimator}
The first local problem estimator  is  designed for
$\bm{Q}_2$ displacement approximation of (\ref{scm11a}). 
The error estimate  
\begin{align*}
\eta_{\mathcal{E}}= \sqrt{\sum_{K\in \mathcal{T}_h }\eta_{\mathcal{E},K}^2}
\end{align*} 
is assembled from estimates of element contributions to the energy error, given by
\begin{align}\label{elastloc1}
\eta_{\mathcal{E},K}^2=2\mu||\bm{\varepsilon}
(\bm{e}_{\mathcal{E},K})||^2_{0,K}+(2\mu)^{-1}||\epsilon_{\mathcal{E},K}||^2_{0,K}
+\lambda^{-1}||\epsilon_{\mathcal{E},K}||^2_{0,K}.
\end{align}
{Specifically, in the cases of $\bm{Q}_2$--$\bm{Q}_{1}$ and 
$\bm{Q}_2$--$\bm{P}_{-1}$ mixed approximations, we
 introduce higher-order {\it correction} spaces $\bm{Q}_3(K)$ and $\bm{Q}_2(K)$ (see~\cite{QLDS}) and solve a mixed elasticity problem  on each element:
find $(\bm{e}_{\mathcal{E},K},\epsilon_{\mathcal{E},K})\in \bm{Q}_3(K)\times\bm{Q}_2(K)$, such that}
\begin{subequations}  \label{elastloc}
\begin{align}\label{elastloc2}
2\mu(\bm{\varepsilon}(\bm{e}_{\mathcal{E},K}),\bm{\varepsilon}(\bm{v}))_K
-( \epsilon_{\mathcal{E},K}, \nabla\cdot \bm{v})_K 
&=(\bm{R}_K,\bm{v})_{K}
\!-\!\!\sum_{E\in\partial K} \! \langle \bm{R}_E, \bm{v}\rangle_E, \quad \forall  \bm{v}\in\bm{Q}_3(K), \\
-(\nabla\cdot \bm{e}_{\mathcal{E},K}, q)_{K} 
-\frac{1}{\lambda}(\epsilon_{\mathcal{E},K}, q)_{K}
&=-(R_K,q)_{K}, \quad \forall q\in \bm{Q}_2(K).\label{elastloc3}
\end{align}
\end{subequations}
The success of this estimation strategy is tied to the stability of the enhanced approximation: we need to  ensure that the local problems are  uniquely solvable in the incompressible limit $\nu = 1/2$. The next two  lemmas provide a more  formal statement.
\begin{lemma}[local inf--sup stability]
Suppose that the space $\bm{Q}_3(K)$ is as defined in \cite{QLDS}, then  there exists a positive constant $\gamma_L$, satisfying
\begin{align}\label{LIS}
\min_{0\neq q_h\in\bm{Q}_2(K)}\max_{0\neq \bm{v}_h\in\bm{Q}_3(K)}
\frac{|(q_h, \nabla\cdot \bm{v}_h)|}{|\bm{v}_h|_1||q_h||_0}\geq \gamma_L 
\end{align}
for all $K\in \mathcal{T}_h$.
\end{lemma}

The estimate (\ref{LIS}) is established in \cite{QLDS} and the next result is a direct consequence. It can be established by following the construction used in establishing Lemma~\ref{Sinsuplem12}.
\begin{lemma}[{local $\mathcal{B}$ stability}]\label{LBS}
For all $(\bm{w},s)\in {\bm{Q}_3(K)}\times\bm{Q}_2(K)$, we have that 
\begin{align}\label{Bstabl}
& \max_{(\bm{v},q)\in {\bm{Q}_3(K)} \times\bm{Q}_2(K)}\frac{\mathcal{B}(\bm{w},s;\bm{v},q)}{(2\mu)^{1/2}
||\bm{\varepsilon}(\bm{v})||_{0,K}+((2\mu)^{-1/2}+\lambda^{-1/2})||q||_{0,K}} \nonumber \\
&\qquad \geq \gamma_B \left( (2\mu)^{1/2}||\bm{\varepsilon}(\bm{w})||_{0,K}+((2\mu)^{-1/2}+\lambda^{-1/2})||s||_{0,K} \right),
\end{align}
where $\gamma_B>0$ is a constant, which only depends on the inf--sup constant $\gamma_L$ in (\ref{LIS}).
\end{lemma}

\begin{theorem}\label{localerrest1}
{In the case of $\bm{Q}_2$--$\bm{Q}_{1}$ or $\bm{Q}_2$--$\bm{P}_{-1}$ 
mixed approximation of  the elasticity problem (\ref{scm11a}), 
the local problem estimator $\eta_{\mathcal{E},K}$ defined  by (\ref{elastloc1})--(\ref{elastloc}) is equivalent  to the local residual error estimator $\eta_{K}$ associated with (\ref{errest1}), 
\begin{align}\label{equiv41}
\eta_{\mathcal{E},K} \lesssim \eta_{K}\lesssim \eta_{\mathcal{E},K} , \quad \forall K\in \mathcal{T}_h.
\end{align}
}
\end{theorem}
{\begin{proof}
Using Lemma \ref{LBS}, we have
\begin{align*}
\eta_{\mathcal{E},K}&=\sqrt{2\mu||\bm{\varepsilon}(\bm{e}_{\mathcal{E},K})||^2_{0,K}+(2\mu)^{-1}||\epsilon_{\mathcal{E},K}|^2_{0,K}+(\lambda)^{-1}||\epsilon_{\mathcal{E},K}||^2_{0,K}}\\
&\le (2\mu)^{1/2}||\bm{\varepsilon}(\bm{e}_{\mathcal{E},K})||_{0,K}+(2\mu)^{-1/2}||\epsilon_{\mathcal{E},K}||_{0,K}+(\lambda)^{-1/2}||\epsilon_{\mathcal{E},K}||_{0,K}\\
&\le \frac{1}{\gamma_B} \max_{(\bm{v},q)\in \bm{Q}_3(K)\times\bm{Q}_2(K)}\frac{\mathcal{B}((\bm{e}_{\mathcal{E},K},\epsilon_{\mathcal{E},K});(\bm{v},q))}{(2\mu)^{\frac{1}{2}}||\bm{\varepsilon}(\bm{v})||_{0,K}+(2\mu)^{-\frac{1}{2}}||q||_{0,K}+(\lambda)^{-\frac{1}{2}}||q||_{0,K}}\\
&\le \frac{1}{\gamma_B} \max_{(\bm{v},q)\in \bm{Q}_3(K)\times\bm{Q}_2(K)}\frac{(\bm{R}_K,\bm{v})-\sum_{E\in\partial K}(\bm{R}_E, \bm{v})_E-(q, R_K)}{(2\mu)^{\frac{1}{2}}||\bm{\varepsilon}(\bm{v})||_{0,K}+(2\mu)^{-\frac{1}{2}}||q||_{0,K}+(\lambda)^{-\frac{1}{2}}||q||_{0,K}}.
\end{align*}
Applying  Cauchy-Schwarz inequality and trace theorem, it follows:
\begin{align*}
\eta_{\mathcal{E},K}&=\sqrt{2\mu||\bm{\varepsilon}(\bm{e}_{\mathcal{E},K})||^2_{0,K}+(2\mu)^{-1}|\epsilon_{\mathcal{E},K}|^2_{0,K}+(\lambda)^{-1}||\epsilon_{\mathcal{E},K}||^2_{0,K}}\le \frac{1}{\gamma_B} \eta_K.
\end{align*}
To prove the upper bound, we use bubble function technique as given in proof of Theorem \ref{efficie}. Define $\bm{w}|_K=\rho^{2}_K \bm{R}_K\, \chi_K $. Using (\ref{bubblea}) and (\ref{elastloc2}), we obtain
\begin{align}\label{resef11}
\eta^2_{R_K}=\rho^2_K ||\bm{R}_K||^2_{0,K}&\lesssim (\bm{R}_K,\rho^2_K \chi_K \bm{R}_K)_{K}\nonumber\\
&= (\bm{R}_K,\bm{w})_{K}=2\mu(\bm{\varepsilon}(\bm{e}_{\mathcal{E},K}),\bm{\varepsilon}(\bm{w}_K))_K-( \epsilon_{\mathcal{E},K}, \nabla\cdot \bm{w}_K)_K.
\end{align} 
Applying Cauchy-Schwarz inequality in (\ref{resef11}) leads to
\begin{align*}
\eta^2_{R_K}\lesssim(\bm{R}_K, \bm{w})_K&\lesssim (2\mu||\bm{\varepsilon}(\bm{e}_{\mathcal{E},K})||^2_{0,K}+(2\mu)^{-1}||\epsilon_{\mathcal{E},K}||^2_{0,K})^{1/2}(2\mu)^{1/2}|\bm{w}_K|_{1,K}.
\end{align*}
Using (\ref{bubblea}) gives
\begin{align*}
\eta^2_{R_K}&\lesssim (2\mu||\bm{\varepsilon}(\bm{e}_{\mathcal{E},K})||^2_{0,K}+(2\mu)^{-1}|\epsilon_{\mathcal{E},K}|^2_{0,K})^{1/2}\eta_{R_K},
\end{align*}
which implies
\begin{align}\label{elastloc11}
\eta^2_{R_K}&\lesssim 2\mu||\bm{\varepsilon}(\bm{e}_{\mathcal{E},K})||^2_{0,K}+(2\mu)^{-1}|\epsilon_{\mathcal{E},K}|^2_{0,K}.
\end{align}
Next, we define
\begin{align*}
\Lambda=\rho_E
\llbracket p_{h}\underline{I}-2\mu\bm{\varepsilon}(\bm{u}_{h})\rrbracket \chi_E.
\end{align*}
where $\chi_E$ is an edge bubble function.
Then from (\ref{elastloc2}), we have
\begin{align*}
\eta_{E_K}^2\lesssim(\llbracket p_{h}\underline{I}-2\mu\bm{\varepsilon}(\bm{u}_{h})\rrbracket,\Lambda)_E=-2\mu(\bm{\varepsilon}(\bm{e}_{\mathcal{E},K}),\bm{\varepsilon}(\Lambda))_K+( \epsilon_{\mathcal{E},K}, \nabla\cdot \Lambda)_K+(R_K,\Lambda).
\end{align*}
Using Cauchy-Schwarz inequality, inequality (\ref{efficie14}) and (\ref{elastloc11}), it follows:
\begin{align*}
\eta_{E_K}^2\lesssim(2\mu||\bm{\varepsilon}(\bm{e}_{\mathcal{E},K})||^2_{0,K}+(2\mu)^{-1}||\epsilon_{\mathcal{E},K}||^2_{0,K})^{1/2}\eta_{E_K},
\end{align*}
which leads to
\begin{align}\label{elastloc12}
\eta_{E_K}^2\lesssim 2\mu||\bm{\varepsilon}(\bm{e}_{\mathcal{E},K})||^2_{0,K}+(2\mu)^{-1}||\epsilon_{\mathcal{E},K}||^2_{0,K}.
\end{align}
Since $\nabla\cdot \bm{u}_h+\frac{1}{\lambda}p_h \in \bm{Q}_2(K)$, then from (\ref{elastloc3}), it holds
\begin{align*}
&(\nabla\cdot \bm{e}_{\mathcal{E},K}, \nabla\cdot \bm{u}_h+\frac{1}{\lambda}p_h)_K +\frac{1}{\lambda}(\epsilon_{\mathcal{E},K}, \nabla\cdot \bm{u}_h+\frac{1}{\lambda}p_h)_K=(R_K,R_K)_K,\\
&||R_K||_{0,K}=||\nabla\cdot\bm{u}_h+\frac{1}{\lambda}p_h||_{0,K}\lesssim ||\nabla\cdot \bm{e}_{\mathcal{E},K}+\frac{1}{\lambda}\epsilon_{\mathcal{E},K}||_{0,K}.
\end{align*}
Using same argument as given in Lemma \ref{efficieR2}, we have
\begin{align}\label{elastloc13}
\eta_{J_K}^2=\rho_d ||R_K||_{0,K}^2&\lesssim \rho_d ||\nabla\cdot \bm{e}_{\mathcal{E},K}+\frac{1}{\lambda}\epsilon_{\mathcal{E},K}||_{0,K}^2\nonumber\\
&\lesssim \rho_d ||\nabla\cdot \bm{e}_{\mathcal{E},K}||_{0,K}^2+\frac{\rho_d}{\lambda^2}||\epsilon_{\mathcal{E},K}||_{0,K}^2\nonumber\\
&\lesssim 2\mu ||\varepsilon (\bm{e}_{\mathcal{E},K})||_{0,K}^2+\frac{1}{\lambda}||\epsilon_{\mathcal{E},K}||_{0,K}^2 \lesssim\eta_{\mathcal{E},K}^2.
\end{align}
Combining  (\ref{elastloc11}),  (\ref{elastloc12}) and  (\ref{elastloc13}) implies the desired upper bound. \quad \end{proof}}
{\begin{remark}
\rbl{All constants arising in the proof of Theorem \ref{localerrest1} are independent of the local mesh parameters ($h_E$ and $h_K$) as well as the Lam\'{e} coefficients ($\mu$ and $\lambda$). Hence the theory confirms that the
 local error estimator $\eta_{\mathcal{E},K}$ is fully robust.}
\end{remark}}
\subsection{{Modified elasticity problem local error estimator}}
The element problem (\ref{elastloc}) that is solved when computing  the elasticity problem local error estimator can be simplified by decoupling the stress tensor term. Thus, rather than solving  (\ref{elastloc}), the  error estimate $\eta_{\mathcal{E},K}$ is computed via (\ref{elastloc1}) using estimates  $\bm{e}_{\mathcal{E},K}$, $\epsilon_{\mathcal{E},K}$ that are generated by solving  the simplified  local problem:
find $(\bm{e}_{\mathcal{E},K},\epsilon_{\mathcal{E},K}) \in \bm{Q}_3(K)\times\bm{Q}_2(K)$, such that
\begin{subequations}  \label{melastloc}
\begin{align}\label{melastloc2}
2\mu(\nabla\bm{e}_{\mathcal{E},K},\nabla\bm{v})_K
-( \epsilon_{\mathcal{E},K}, \nabla\cdot \bm{v})_K 
&=(\bm{R}_K,\bm{v})_{K}
\!-\!\!\sum_{E\in\partial K} \! \langle \bm{R}_E, \bm{v}\rangle_E, \quad \forall  \bm{v}\in\bm{Q}_3(K), \\
-(\nabla\cdot \bm{e}_{\mathcal{E},K}, q)_{K} 
-\frac{1}{\lambda}(\epsilon_{\mathcal{E},K}, q)_{K}
&=-(R_K,q)_{K}, \quad \forall q\in \bm{Q}_2(K),\label{melastloc3}
\end{align}
\end{subequations}
where the correction spaces $\bm{Q}_3(K)$ and $\bm{Q}_2(K)$ are unchanged.
The stability of the  simplified local problem (\ref{melastloc}) can be easily checked: the   
bilinear form $\mathcal{B}$ in (\ref{Bstabl}) is simply replaced by the decoupled variant 
$$ \mathcal{B}_{\mathcal{ME}}(\bm{w},s;\bm{v},q)= 2\mu(\nabla\bm{w},\nabla\bm{v})_K
- ( s, \nabla\cdot \bm{v})_K-(\nabla\cdot \bm{w}, q)
 -\frac{1}{\lambda}(s, q).
 $$
We omit the details. An equivalence  result is formally stated below.
\begin{theorem}
{In the case of $\bm{Q}_2$--$\bm{Q}_{1}$ or $\bm{Q}_2$--$\bm{P}_{-1}$ 
mixed approximation of  the elasticity problem (\ref{scm11a}), 
the local problem estimator $\eta_{\mathcal{E},K}$ defined  by (\ref{elastloc1}) and (\ref{melastloc}) is equivalent  to the local residual error estimator $\eta_{K}$ associated with (\ref{errest1}), 
\begin{align}\label{equiv42}
\eta_{\mathcal{E},K} \lesssim \eta_{K}\lesssim \eta_{\mathcal{E},K} , \quad \forall K\in \mathcal{T}_h.
\end{align}
}
\end{theorem}
{The next estimator simplifies the local problem solved in (\ref{melastloc})  even further.}

\subsection{{Stokes problem local error estimator}} \label{sec_Stokes}
The Stokes problem estimator
\begin{align*}
\eta_{S} = \sqrt{\sum_{K \in {\cal T}_{h}} \eta_{S,K}^{2}}
\end{align*} 
can be assembled from local contributions given by 
\begin{align}\label{stokesloc1}
\eta_{S,K}^2=2\mu|\bm{e}_{S,K}|^2_{1,K}+\rho_d^{-1}||\epsilon_{S,K}||^2_{0,K},
\end{align}
where $(\bm{e}_{S,K},\epsilon_{S,K})\in \bm{Q}_3(K)\times\bm{Q}_2(K)$ 
solves a Stokes problem on each element:
\begin{subequations}  \label{stokesloc}
\begin{align}\label{stokesloc2}
2\mu(\nabla\bm{e}_{\mathcal{S},K},\nabla\bm{v})_K
-( \epsilon_{\mathcal{S},K}, \nabla\cdot \bm{v})_K 
&=(\bm{R}_K,\bm{v})_{K}
\!-\!\!\sum_{E\in\partial K} \! \langle \bm{R}_E, \bm{v}\rangle_E, \quad \forall  \bm{v}\in\bm{Q}_3(K), \\
-(\nabla\cdot \bm{e}_{\mathcal{S},K}, q)_{K} 
&=-(R_K,q)_{K}, \quad \forall q\in \bm{Q}_2(K).\label{stokesloc3}
\end{align}
\end{subequations}
The stability of this approximation is also an immediate consequence of the inf--sup stability \eqref{LIS} of the chosen correction spaces.

\begin{lemma}[{local $\mathcal{B}_S$ stability}]\label{SLBS1}
For all $(\bm{w},s)\in  \bm{Q}_3(K) \times\bm{Q}_2(K)$, we have that 
\begin{align}
&\max_{(\bm{v},q)\in {\bm{Q}_3(K)} \times\bm{Q}_2(K)}
\frac{\mathcal{B}_S(\bm{w},s;\bm{v},q)}{(2\mu)^{1/2}|\bm{v}|_{1,K}+\rho_d^{-1/2}||q||_{0,K}}\nonumber\\
&\qquad\geq \gamma_{S} \left( (2\mu)^{1/2}|\bm{w}|_{1,K}+\rho_d^{-1/2}||s||_{0,K}\right),
\end{align}
where $\mathcal{B}_{S}(\bm{w},s;\bm{v},q)= 2\mu(\nabla\bm{w},\nabla\bm{v})_K-(s, \nabla\cdot \bm{v})_K-(\nabla\cdot \bm{w}, q) $. The stability  constant $\gamma_{S}>0$ only depends on the inf--sup constant $\gamma_L$ in (\ref{LIS}).
\end{lemma}

\noindent 
Once again, we omit the details. An equivalence  result is formally stated below.
\smallskip 

\begin{theorem}
In the case of $\bm{Q}_2$--$\bm{Q}_{1}$ or $\bm{Q}_2$--$\bm{P}_{-1}$ 
mixed approximation of  the elasticity problem (\ref{scm11a}), 
the local problem estimator $\eta_{\mathcal{S},K}$ defined  by (\ref{stokesloc1})--(\ref{stokesloc}) is equivalent  to the modified elasticity error estimator $\eta_{\mathcal{E},K}$ defined  by (\ref{elastloc1}) and (\ref{melastloc}), 
\begin{align}\label{equiv43}
\eta_{\mathcal{S},K} \lesssim \eta_{\mathcal{E},K} \lesssim \eta_{\mathcal{S},K} , \quad \forall K\in \mathcal{T}_h.
\end{align}
\end{theorem}
\noindent 
The last estimator we consider simplifies the local problem in (\ref{elastloc}) in a clever  way.

\subsection{{Poisson problem local error estimator}} \label{sec_Poisson}
The Poisson problem estimator 
\begin{align*}
\eta_{P} = \sqrt{\sum_{K \in {\cal T}_{h}} \eta_{P,K}^{2}}
\end{align*}
is assembled from local contributions given by
\begin{align}\label{poiloc1}
\eta_{P,K}^2=2\mu|\bm{e}_{P,K}|^2_{1,K}+\rho_d^{-1}||\epsilon_{P,K}||^2_{0,K},
\end{align}
where $(\bm{e}_{P,K}, \epsilon_{P,K})\in \bm{Q}_3(K)\times\bm{Q}_2(K)$ 
solve the following problem  on each element: 
\begin{subequations}  \label{poiloc}
\begin{align}\label{poiloc2}
2\mu \, (\nabla\bm{e}_{P,K},\nabla\bm{v})_K &=(\bm{R}_K,\bm{v})_K
-\sum_{E\in\partial K} \langle \bm{R}_E, \bm{v} \rangle_E, \quad \forall \bm{v}\in\bm{Q}_3(K), \\
\rho_d^{-1}(\epsilon_{P,K}, q)_K &=(R_K,q)_K, \quad \forall q\in \bm{Q}_2(K).\label{poiloc3}
\end{align}
\end{subequations}  
This local problem is very appealing from a computational perspective.  First,  (\ref{poiloc2}) decouples into a pair of local Poisson problems. Second, since $\epsilon_{P,K}\in \bm{Q}_2(K)$, the solution of (\ref{poiloc3}) is immediate: 
$\epsilon_{P,K} = \rho_d R_K = \rho_d (\nabla\cdot \bm{u}_h + \lambda^{-1} p_h)$, 
thus the local error estimator (\ref{poiloc1}) simplifies to
$   \eta_{P,K}^2 =2\mu \, ||\nabla \bm{e}_{P,K}||^2_{0,K} 
+ \rho_d || \nabla\cdot \bm{u}_h + \lambda^{-1} p_h ||^2_{0,K}$.
A third attractive feature of this decoupled approach is that the stability of the local problem (\ref{poiloc}) is guaranteed---there is no need to construct compatible correction spaces. An equivalence result for the error estimator is formally stated below.
\begin{theorem}\label{poiest11}
The local problem estimator $\eta_{P,K}$ defined by \eqref{poiloc1} and \eqref{poiloc} is equivalent to the modified 
elasticity  error estimator $\eta_{\mathcal{E},K}$ defined  by (\ref{elastloc1}) and (\ref{melastloc}), 
\begin{align}\label{equiv44}
\eta_{\mathcal{P},K} \lesssim \eta_{\mathcal{E},K} \lesssim \eta_{\mathcal{P},K} , 
\quad \forall K\in \mathcal{T}_h.
\end{align} 
\end{theorem}
{\begin{proof}
From (\ref{elastloc2}), (\ref{elastloc3}), (\ref{poiloc2}) and (\ref{poiloc3}), for any $K\in \mathcal{T}_h$ and $(\bm{v},q)\in \bm{Q}_3(K)\times\bm{Q}_2(K)$, we have that
\begin{align}
2\mu \, (\nabla\bm{e}_{P,K},\nabla\bm{v})_K -\rho_d^{-1}(\epsilon_{P,K}, q)_K&=(\bm{R}_K\,\bm{v})-\sum_{E\in\partial K}(\bm{R}_E, \bm{v})_E-(R_K,q)\nonumber\\
&=2\mu(\bm{\varepsilon}(\bm{e}_{\mathcal{E},K}),\bm{\varepsilon}(\bm{v}))_K-( \epsilon_{\mathcal{E},K}, \nabla\cdot \bm{v})_K\nonumber\\
&\quad-(\nabla\cdot \bm{e}_{\mathcal{E},K}, q) -\frac{1}{\lambda}(\epsilon_{\mathcal{E},K}, q)\nonumber\\
&=\mathcal{B}((\bm{e}_{\mathcal{E},K},\epsilon_{\mathcal{E},K});(\bm{v},q)).
\end{align}
Using the local $\mathcal{B}$ stability from Lemma \ref{LBS}, it holds:
\begin{align*}
 &(2\mu)^{1/2}||\bm{\varepsilon}(\bm{e}_{\mathcal{E},K})||_{0,K}+(2\mu)^{-1/2}||\epsilon_{\mathcal{E},K}||_{0,K}+(\lambda)^{-1/2}||\epsilon_{\mathcal{E},K}||_{0,K}\\
&\le \frac{1}{\gamma_B} \max_{(\bm{v},q)\in \bm{Q}_3(K)\times\bm{Q}_2(K)}\frac{\mathcal{B}((\bm{e}_{\mathcal{E},K},\epsilon_{\mathcal{E},K});(\bm{v},q))}{(2\mu)^{1/2}||\bm{\varepsilon}(\bm{v})||_{0,K}+(2\mu)^{-1/2}||q||_{0,K}+(\lambda)^{-1/2}||q||_{0,K}}\\
&=\frac{1}{\gamma_B} \max_{(\bm{v},q)\in \bm{Q}_3(K)\times\bm{Q}_2(K)}\frac{2\mu \, (\nabla\bm{e}_{P,K},\nabla\bm{v})_K -\rho_d^{-1}(\epsilon_{P,K}, q)_K}{(2\mu)^{1/2}||\bm{\varepsilon}(\bm{v})||_{0,K}+(2\mu)^{-1/2}||q||_{0,K}+(\lambda)^{-1/2}||q||_{0,K}}.
\end{align*}
Applying Cauchy-Schwarz inequality implies that
 \begin{align}
 (2\mu)^{1/2}||\bm{\varepsilon}(\bm{e}_{\mathcal{E},K})||_{0,K}+(2\mu)^{-1/2}|\epsilon_{\mathcal{E},K}|_{0,K}+(\lambda)^{-1/2}|\epsilon_{\mathcal{E},K}|_{0,K}
&\lesssim \eta_{P,K}.
 \end{align}
 To establish the upper bound, we take $\bm{v}\in \bm{Q}_3(K)$, and then using (\ref{elastloc2}) and (\ref{poiloc2}) leads to 
 \begin{align}\label{sygloc4}
 2\mu(\nabla\bm{e}_{P,K},\nabla\bm{v})_K &=(\bm{R}_K\,\bm{v})-\sum_{E\in\partial K}(\bm{R}_E, \bm{v})_E\nonumber\\
 &=2\mu(\bm{\varepsilon}(\bm{e}_{\mathcal{E},K}),\bm{\varepsilon}(\bm{v}))_K-( \epsilon_{\mathcal{E},K}, \nabla\cdot \bm{v})_K.
 \end{align}
 From (\ref{elastloc3}) and (\ref{poiloc3}), it holds:
 \begin{align}\label{sygloc5}
 \rho_d^{-1}(\epsilon_{P,K}, q)_K&=(R_K,q)=(\nabla\cdot \bm{e}_{\mathcal{E},K}, q) +\frac{1}{\lambda}(\epsilon_{\mathcal{E},K}, q), \quad \forall q\in \bm{Q}_2(K).
 \end{align}
 Using (\ref{sygloc4}) gives
 \begin{align}
 (2\mu)^{1/2}||\nabla\bm{e}_{P,K}||_{0,K}&=\max_{v\in\bm{Q}_3(K)}\frac{2\mu(\nabla\bm{e}_{P,K},\nabla\bm{v})}{(2\mu)^{1/2}||\nabla\bm{v}||_{0,K}}\nonumber\\
 &=\max_{v\in\bm{Q}_3(K)}\frac{2\mu(\bm{\varepsilon}(\bm{e}_{\mathcal{E},K}),\bm{\varepsilon}(\bm{v}))_K-( \epsilon_{\mathcal{E},K}, \nabla\cdot \bm{v})_K}{(2\mu)^{1/2}||\bm{\varepsilon}(\bm{v})||_{0,K}}.
 \end{align}
Applying Cauchy-Schwarz inequality leads to
\begin{align}\label{sygloc6}
 (2\mu)^{1/2}||\nabla\bm{e}_{P,K}||_{0,K}\lesssim (2\mu)^{1/2}||\bm{\varepsilon}(\bm{e}_{\mathcal{E},K})||_{0,K}+(2\mu)^{-1/2}||\epsilon_{\mathcal{E},K}||_{0,K}.
 \end{align}
Using (\ref{sygloc5}), we have 
 \begin{align}
\rho_d^{-1/2}||\epsilon_{P,K}||_{0,K}&=\max_{q\in\bm{Q}_2(K)}\frac{\rho_d^{-1}(\epsilon_{P,K},q)}{\rho_d^{-1/2}||q||_{0,K}}\nonumber\\
&=\max_{q\in\bm{Q}_2(K)}\frac{(\nabla\cdot \bm{e}_{\mathcal{E},K}, q)_K +\frac{1}{\lambda}(\epsilon_{\mathcal{E},K}, q)_K}{\rho_d^{-1/2}||q||_{0,K}}\nonumber\\
 &\le \max_{q\in\bm{Q}_2(K)}\frac{\rho_d^{1/2}||\nabla\cdot \bm{e}_{\mathcal{E},K}+\frac{1}{\lambda}\epsilon_{\mathcal{E},K}||_{0,K}\rho_d^{-1/2}||q||_{0,K}}{\rho_d^{-1/2}||q||_{0,K}}\nonumber\\
 &\le \rho_d^{1/2}||\nabla\cdot \bm{e}_{\mathcal{E},K}+\frac{1}{\lambda}\epsilon_{\mathcal{E},K}||_{0,K}.
 \end{align}
 By (\ref{elastloc13}), it holds:
 \begin{align}\label{sygloc7}
 \rho_d^{-1}||\epsilon_{P,K}||_{0,K}^2&\le \rho_d ||\nabla\cdot \bm{e}_{\mathcal{E},K}+\frac{1}{\lambda}\epsilon_{\mathcal{E},K}||_{0,K}^2 \nonumber\\
 &\lesssim (2\mu)||\bm{\varepsilon}(\bm{e}_{\mathcal{E},K})||_{0,K}^2+(\lambda)^{-1}||\epsilon_{\mathcal{E},K}||_{0,K}^2.
 \end{align}
 Combining (\ref{sygloc6}) and (\ref{sygloc7}) implies the required upper bound. 
\end{proof}}
{\begin{remark}
\rbl{All constants arising in the proof of Theorem  \ref{poiest11} are independent of the local mesh parameters ($h_E$ and $h_K$) as well as the Lam\'{e} coefficients ($\mu$ and $\lambda$). Hence the theory confirms that the
 local error estimator $\eta_{P,K}$  is fully robust.} 
\end{remark}}

We note that the strategy of decoupling the components of local problem error estimators in a mixed setting was pioneered by Ainsworth \& Oden, see~\cite[Section\,9.2]{MJ}.

\section{{Computational results}}\label{Numres}
In this concluding section, we present computational results for three test problems in order to critically compare the performance of some of the error estimation strategies introduced above. Specifically, results are shown  for the residual estimator $\eta$ defined in \eqref{errest1}, the Stokes problem local estimator $\eta_S$ defined in Section \ref{sec_Stokes}, and the Poisson problem local estimator $\eta_P$ defined in Section \ref{sec_Poisson}. In all the reported results, the mixed finite element approximation $(\bm{u}_{h}, p_{h})$ was computed using $\bm{Q}_2$--$\bm{Q}_{1}$ elements, using the IFISS\footnote{IFISS is an open-source MATLAB toolbox that includes algorithms for discretization by mixed finite element methods and a posteriori error estimation of computed solutions; see~\cite{ifiss} for a review\rbl{.}} toolbox~\cite{DHA}. Error estimates computed for $\bm{Q}_2$--$\bm{P}_{-1}$ mixed approximations (that is, with a discontinuous linear pressure) showed exactly the same behaviour, but are not reported.

\subsection{Analytic solution}
Our first problem has been used as a test problem by Carstensen \& Gedicke~\cite{CJ}. It is posed on a square domain $ \Omega = (0,1)\times (0,1)$ with a zero essential boundary condition on $\Gamma_{D}=\Gamma$. The load is $\textbf{f}=(f_1,f_2)^{\top}$  where
\begin{align*}
f_1&=-2\mu\pi^3\cos(\pi y)\sin(\pi y)(2\cos(2\pi x) -1), \\
f_2&= 2\mu\pi^3\cos(\pi x)\sin(\pi x)(2\cos(2\pi y)-1),
\end{align*}
and the corresponding exact displacement vector is given by $\textbf{u}=(u_1,u_2)^{\top}$
where
\begin{align*}
u_1&=\pi\cos(\pi y)\sin^2(\pi x)\sin(\pi y),\quad
u_2= -\pi\cos(\pi x)\sin^2(\pi y)\sin(\pi x).
\end{align*}
Here, the exact pressure solution is $p=0$, since $\nabla \cdot \bm{u}=0$. The mixed finite element approximation is computed using uniform grids of square elements of width $h$.

Figure~\ref{Ex1mu100} shows the convergence behaviour of the energy norm of the exact error $e=|||(\bm{u}-\bm{u}_{h},p-p_{h})|||$ as the finite element mesh is refined, as well as the estimates obtained with the three chosen error estimation strategies. In this experiment, $\mu$ is fixed,  and we consider two representative values of the Poisson ratio $\nu$. In both cases, the energy norm error converges to zero at the optimal rate for a smooth solution; that is both the exact and estimated errors are $\mathcal{O}(h^2)$. The fact that the lines on the error plots associated with the estimators $\eta$, $\eta_P$ and $\eta_S$ all lie parallel to the line associated with the exact error confirms that all three estimators are  efficient as well as being reliable. Moreover, the results clearly indicate that the Poisson problem local error estimator is the method of choice:  $\eta_P$ is relatively cheap to compute and gives more accurate estimates of the error $e$ than the residual estimator $\eta$. The computed effectivity indices displayed in Table~\ref{tabeff100} reinforce this point. The effectivity of the Poisson problem local estimator is close to unity even when $\nu$ approaches the incompressible limit. Identical effectivity indices to those shown are generated when the experiments are repeated with smaller  values of $\mu$:  specifically, we tested $\mu=1$ and  $\mu=0.01$. All our results confirm the robustness of the three error estimators to variations in the parameters $\mu$ and $\nu$, and hence $\lambda$ (since $\lambda= 2 \mu \nu / (1-2\nu)$).

\begin{figure}[!ht]
\centering
\includegraphics[width=.51\textwidth]{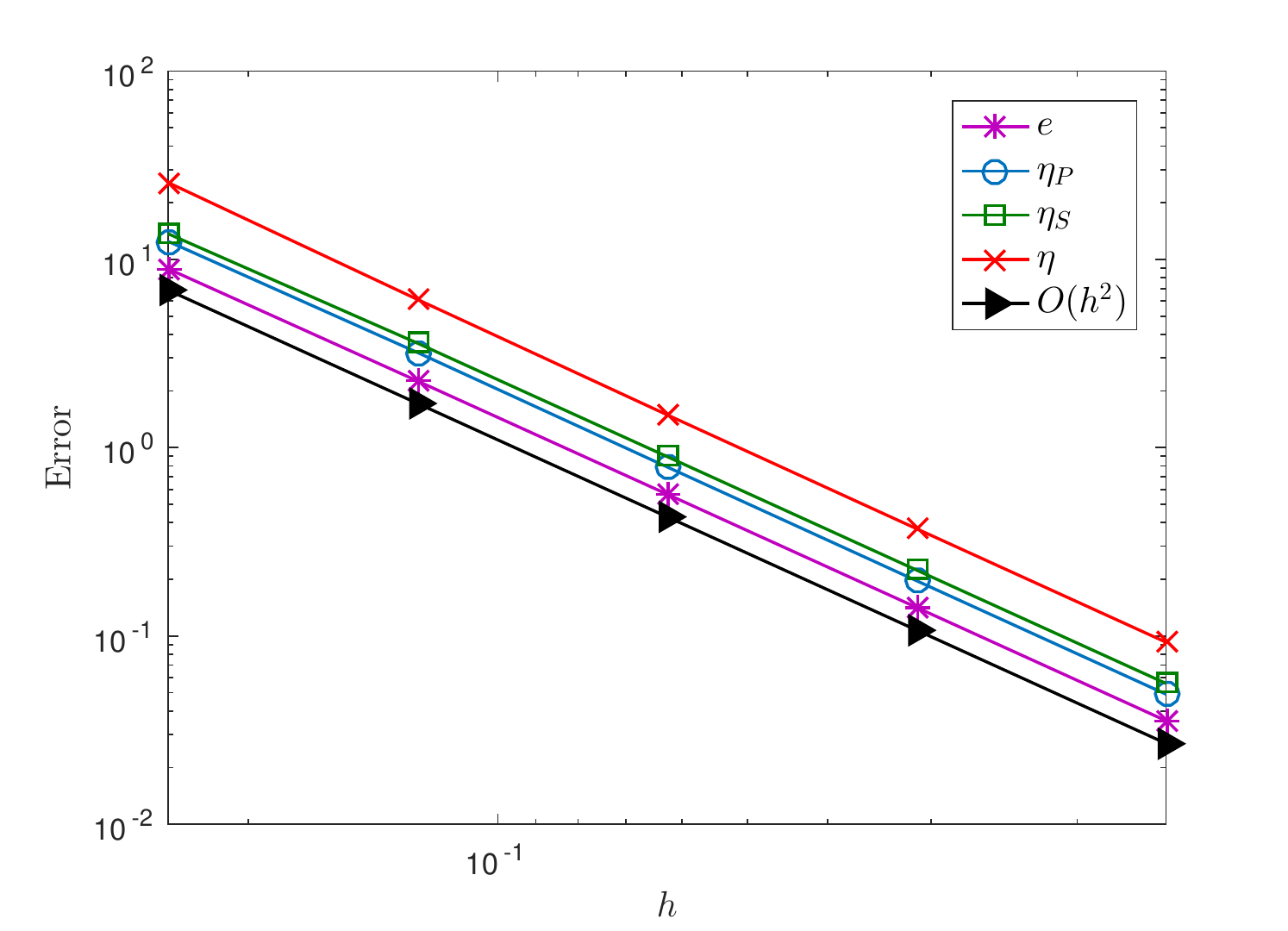}
\label{ex1mu100nu4}
\centering
\includegraphics[width=.46\textwidth]{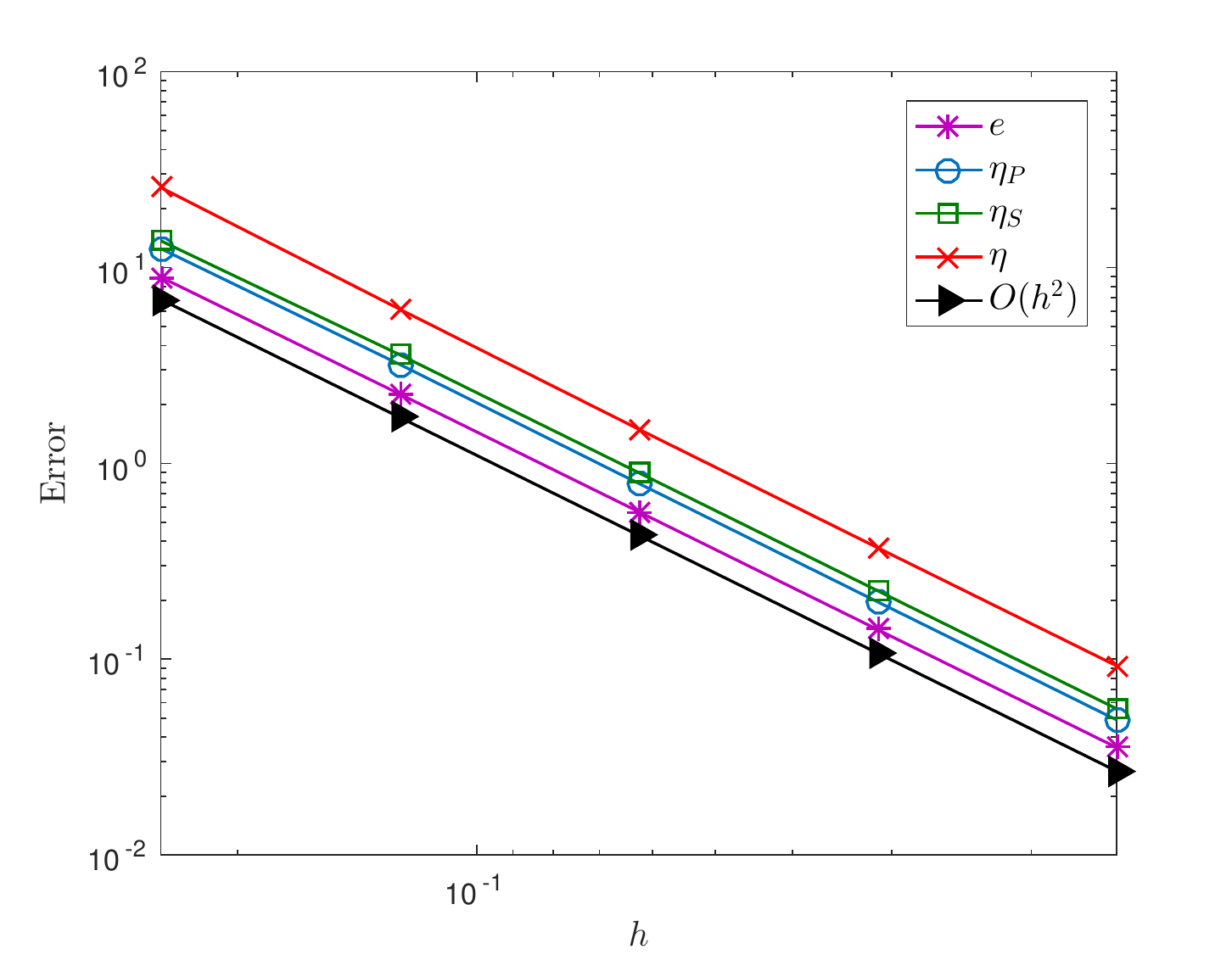}
\label{ex1mu100nu49999}
\caption{Exact $(e)$ and estimated energy errors computed with the residual-based ($\eta$), local Stokes problem $(\eta_{S})$ and local Poisson problem ($\eta_{P}$) estimators, for varying mesh size $h$ and parameters $\mu=100$ and $\nu=.4$ (left); $\mu=100$ and $\nu=0.49999$ (right).}
\label{Ex1mu100}
\end{figure}

 \begin{table}[ht!]
 \caption{Effectivity indices computed with the residual-based ($\eta$), local Stokes problem $(\eta_{S})$ and local Poisson problem ($\eta_{P}$) estimators, for varying Poisson ratios $\nu=.4, .499,.49999$ and fixed $\mu=100$.}
    \label{tabeff100}
\begin{center}
    \begin{tabular}{ | c | c | c | c | c | c | c | c | c | c | }
    \hline
    \multicolumn{4}{| c |}{$\nu=.4$}&\multicolumn{3}{| c |}{$\nu=.499$}&\multicolumn{3}{| c |}{$\nu=.49999$}\\
    \hline
    $h$ & $\frac{\eta}{e}$& $\frac{\eta_{\mathcal{S}}}{e}$& $\frac{\eta_{\mathcal{P}}}{e}$& $\frac{\eta}{e}$& $
    \frac{\eta_{\mathcal{S}}}{e}$& $\frac{\eta_{\mathcal{P}}}{e}$& $\frac{\eta}{e}$& $\frac{\eta_{\mathcal{S}}}{e}$& $
    \frac{\eta_{\mathcal{P}}}{e}$\\
    \hline   
     $\frac{1}{4}$ &$2.850$ &$1.5197$ &$ 1.3808$ &$2.847$ &$1.5176 $ & $1.3794$ &$2.847$ &$ 1.5175$ & $ 1.3794$\\
   $ \frac{1}{8}$ &$ 2.701$  &$ 1.5799$ &$1.4071$ & $  2.701$&$1.5797$&  $ 1.4070$ &$2.701$ &$1.5797 $ & $ 1.4070$ \\
    $\frac{1}{16}$ &$ 2.636$  & $ 1.5804$& $  1.3919$&$2.636$ & $ 1.5804 $& $ 1.3919$ &$ 2.636$ &$ 1.5804$ & $ 1.3919$\\
    $\frac{1}{32}$ &$2.617$  & $1.5782$& $ 1.3850$& $ 2.617$& $1.5782$& $ 1.3850$ &$2.617$ &$1.5782$ & $1.3850$ \\
    $\frac{1}{64}$ &$2.612 $  & $1.5774$& $1.3830$& $ 2.612 $& $1.5774$& $1.3830$ &$2.612$ &$1.5774 $ & $ 1.3830$ \\
    \hline
    \end{tabular}   
\end{center}

\end{table}

\subsection{Nonsmooth pressure solution}
The second problem is taken from  Houston et al.~\cite{PDT} and is posed on a square domain $\Omega = (0,1)\times (0,1)$. There is no body force, so $\bm{f}=\bm{0}$, but there is a nonzero essential boundary condition. We have $\bm{u}= (g,0)^{\top}$ on $\Gamma_{D}=\Gamma$ (so $\Gamma_{N}=\emptyset$), where
\begin{align*}
 g=\Big\{\begin{array}{lr}
\sin^2(\pi x), &\mbox{for}\; y=1,\\
 0,& \hbox{elsewhere}.
\end{array} 
\end{align*}

\begin{figure}[!th]
\centering
\includegraphics[width=.75\textwidth]{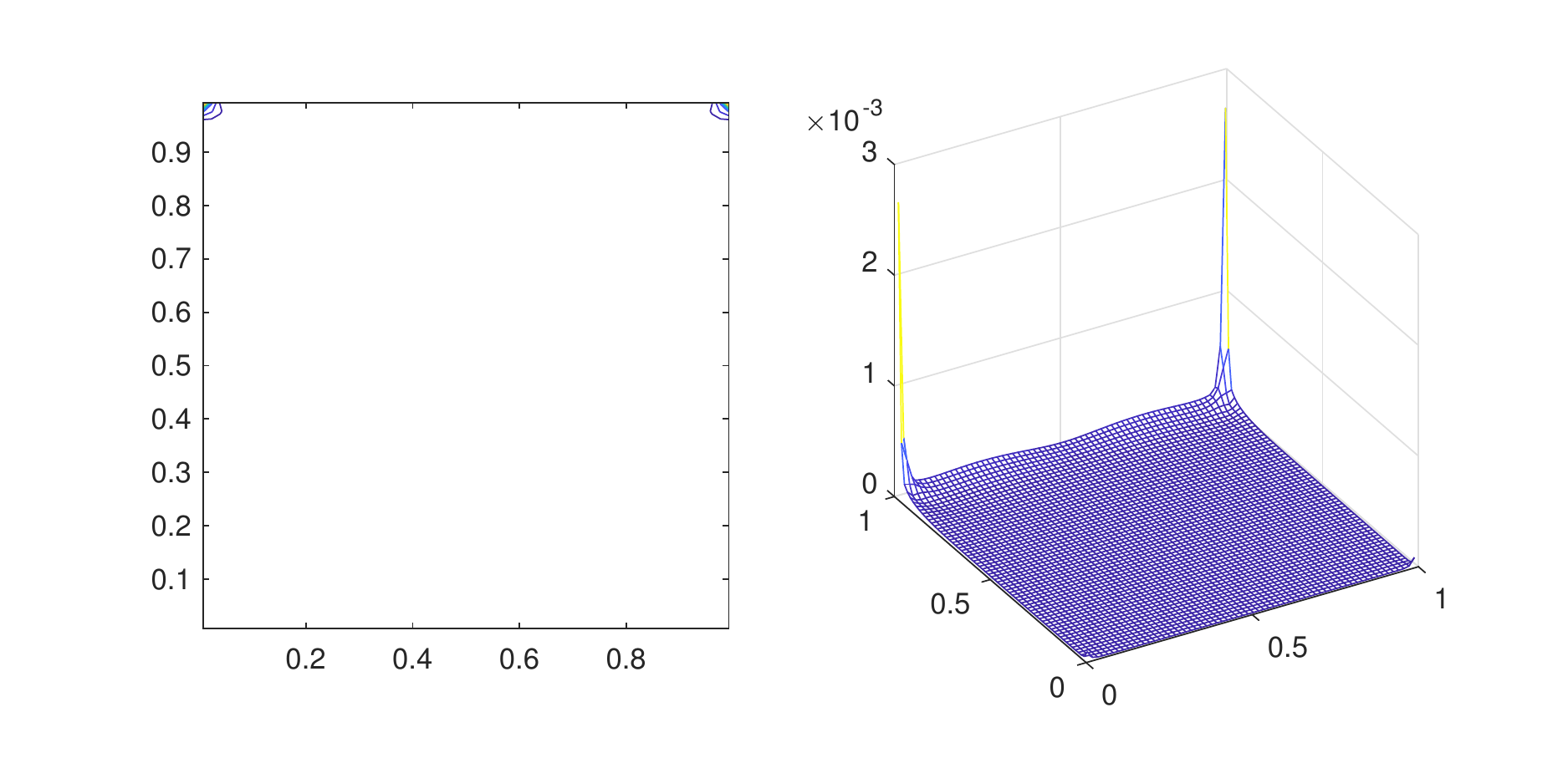}
\label{exc2nu4PEE}
\vspace{0.1in} 
\includegraphics[width=.75\textwidth]{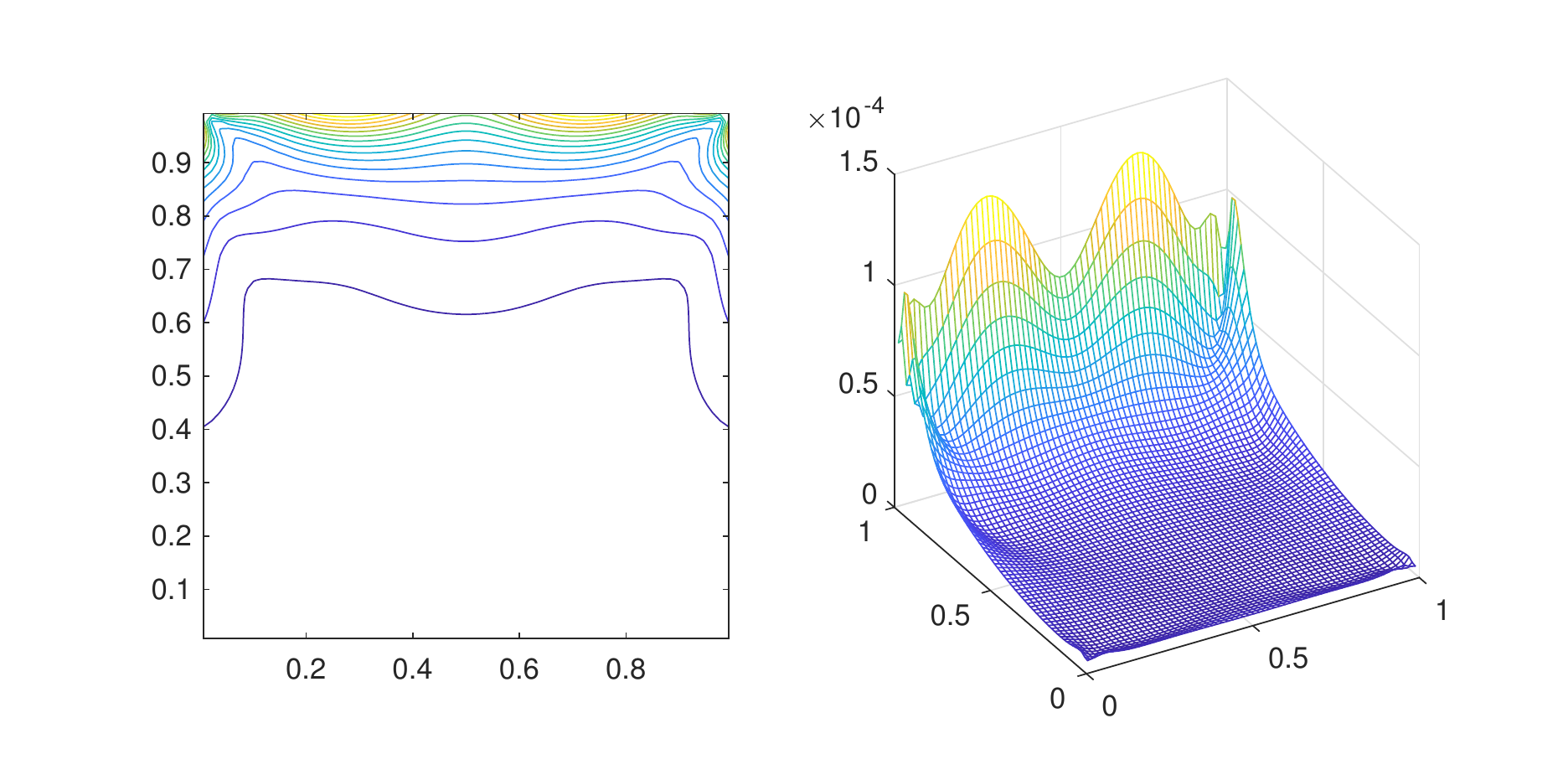}
\label{exc2nu49999PEE}
\caption{Contour and mesh plots of the element contributions $\eta_{P,K}$ to the local Poisson estimator $\eta_{P}$ for the second test problem, with Poisson ratio $\nu=0.4$ (top) and $\nu=0.4999$ (bottom).}
\label{Ex2nu4}
\end{figure}

This is a challenging problem if one is trying to solve it using a standard (not mixed) formulation of the planar elasticity equations, due to the locking phenomenon that occurs when $\nu \to 1/2$. In the mixed formulation, there are  pressure singularities at the top corners of the domain, but these become insignificant in the incompressible limit. As one would expect, the singular behaviour is detected by all three error estimators---it can be clearly seen in the comparison of the estimated errors computed using the Poisson problem local estimator $\eta_{P}$ for two different values of $\nu$ shown in Figure~\ref{Ex2nu4}. Note that while the solution exhibits full $H^2$-regularity, it is not $H^3$-regular. This lack of smoothness is reflected in the observed convergence rate of the estimated energy norm errors obtained with all three estimators; see Figure~\ref{Ex2c2}. Our computational results suggest that the energy norm error converges to zero at a suboptimal rate. The error is estimated to be $\mathcal{O}(h^{1.6})$ when $\nu=0.4$, but we also see that the optimal rate of two is recovered when sufficiently close to the incompressible limit $\nu=1/2$.

\begin{figure}[!th]
\centering
\includegraphics[width=.49\textwidth]{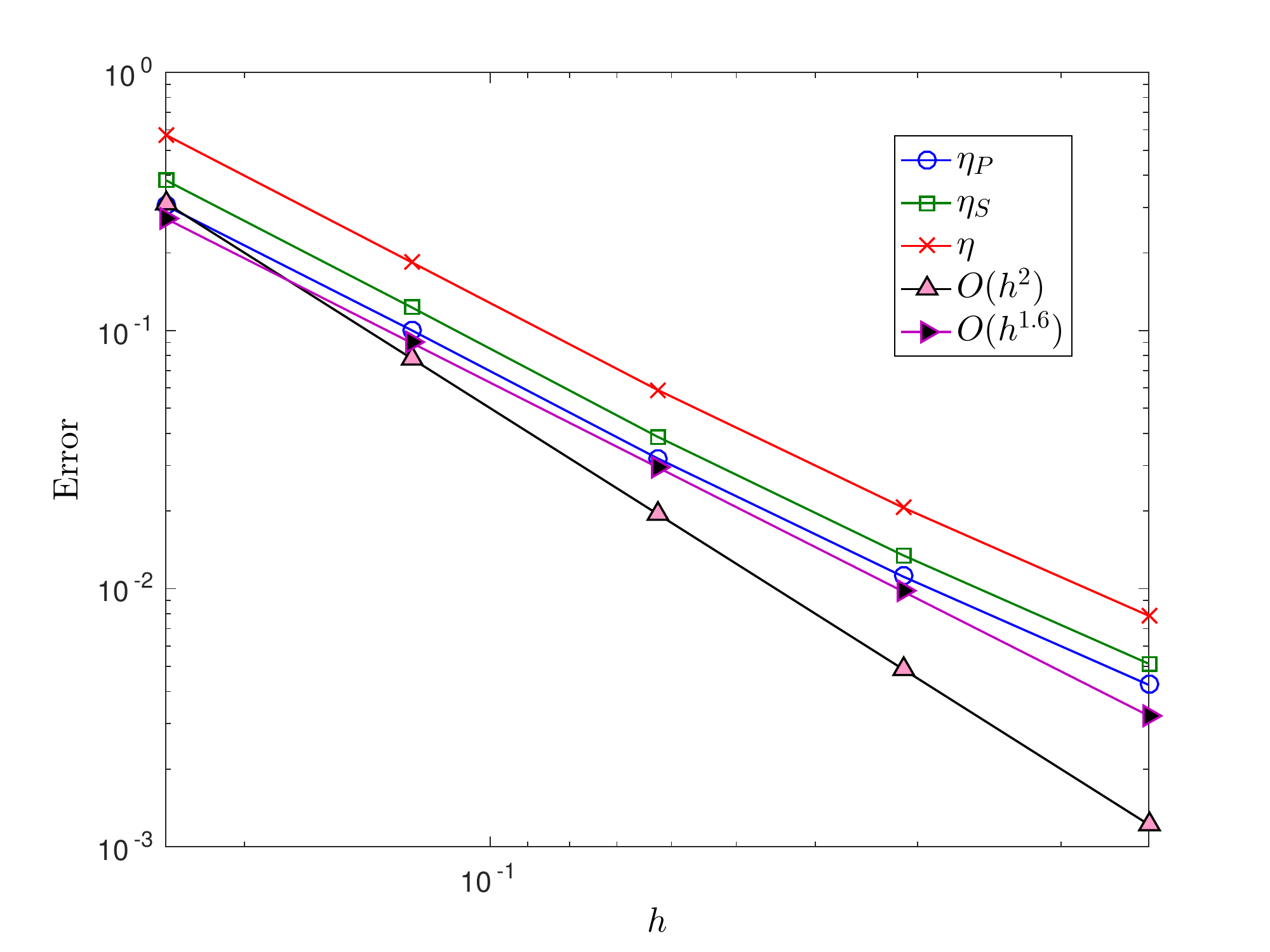}
\label{exc2nu4}
\centering
\includegraphics[width=.49\textwidth]{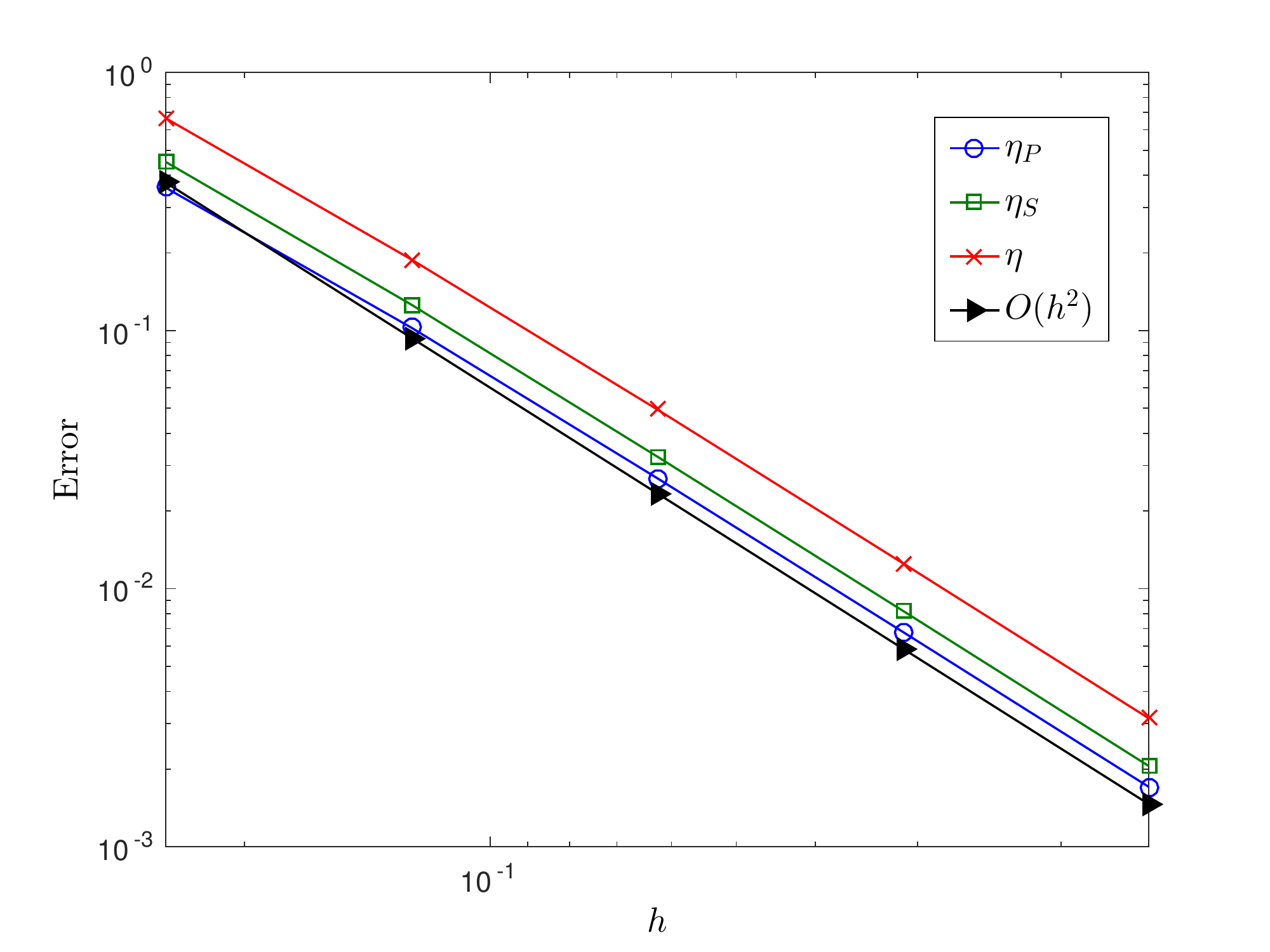}
\label{exc2nu49999}
\caption{Estimated energy errors computed with the residual-based ($\eta$), local Stokes problem $(\eta_{S})$ and local Poisson problem ($\eta_{P}$) estimators, for varying mesh size $h$ and parameters $\mu=1$ and $\nu=.4$ (left); $\mu=1$  and $\nu=0.49999$  (right).}
\label{Ex2c2}
\end{figure}

\subsection{Mixed boundary conditions}
To test the error estimation strategies on a more realistic example, we extend the second problem above to include a natural boundary condition (so that $\Gamma_{N} \neq \emptyset$).
Specifically, we now consider the square domain $\Omega = (-1,1)\times (-1,1)$, with a natural condition on the right edge $\Gamma_{N}=\left\{1\right\} \times (-1,1).$  \rbl{We impose a zero essential boundary condition on $\Gamma_{D}=\Gamma \setminus \Gamma_{N}$ and we also set $\bm{f}=(1,1)^\top$.}

\begin{figure}
\centering
\includegraphics[width=.46\textwidth]{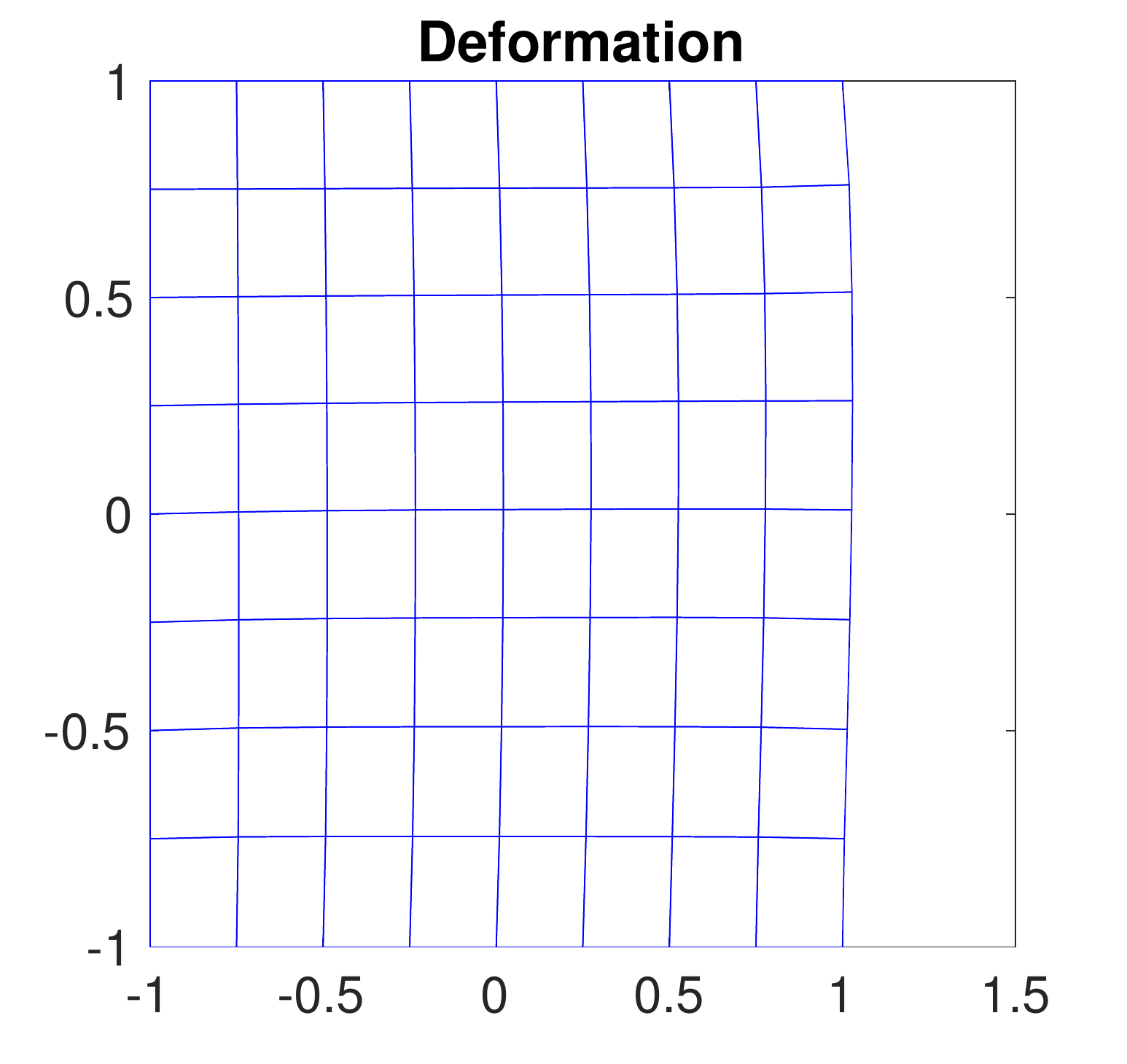}
\label{exc2bnu4DEF}
\centering
\includegraphics[width=.46\textwidth]{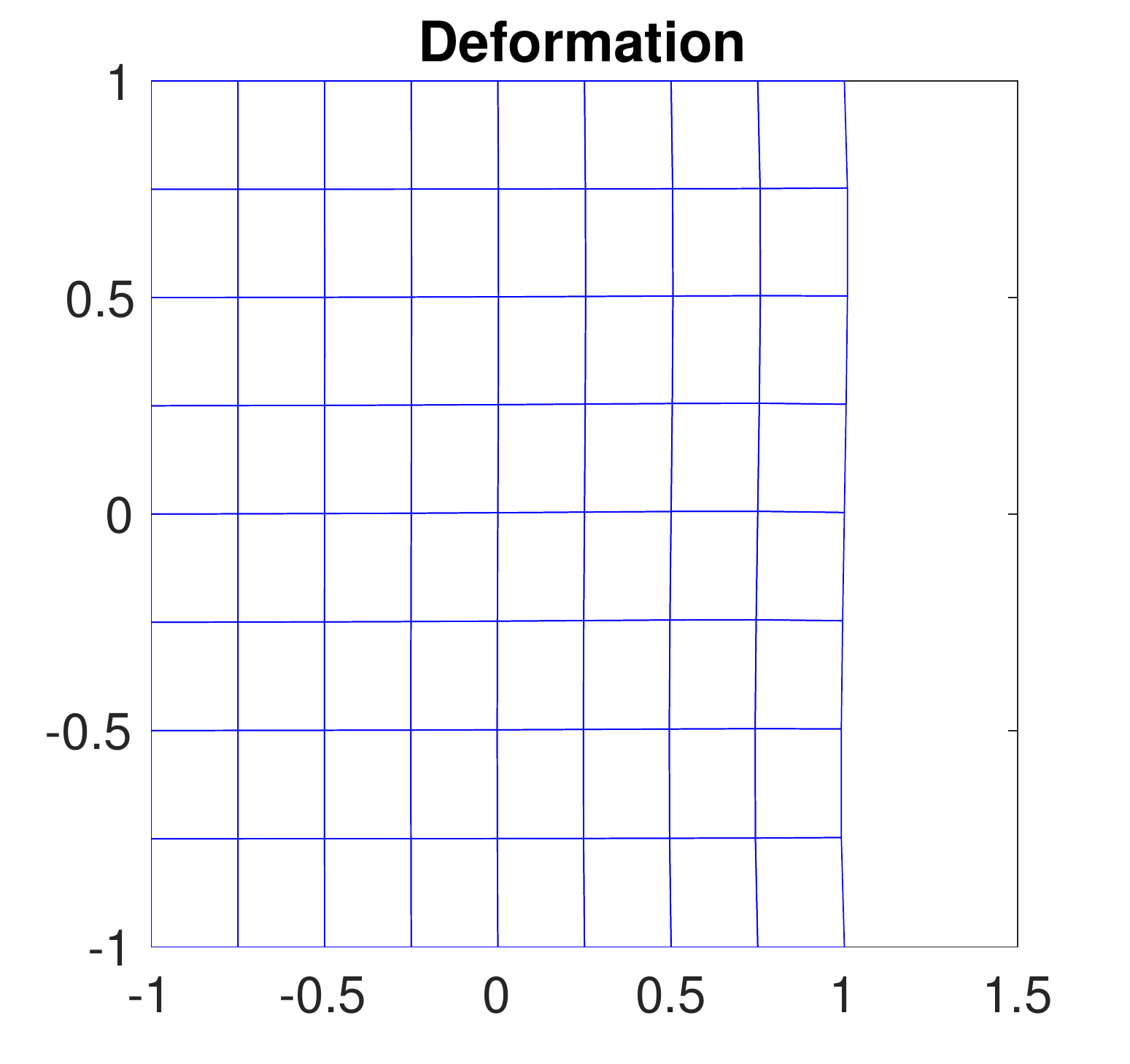}
\label{exc2bnu49999DEF}
\caption
{Computed deformation for $h=1/4$, with $\mu=10$ and $\nu=.4$ (left); $\nu=0.49999$ (right).}
\label{Ex1c2bDEF}
\end{figure}

\begin{figure}
\centering
\includegraphics[width=.46\textwidth]{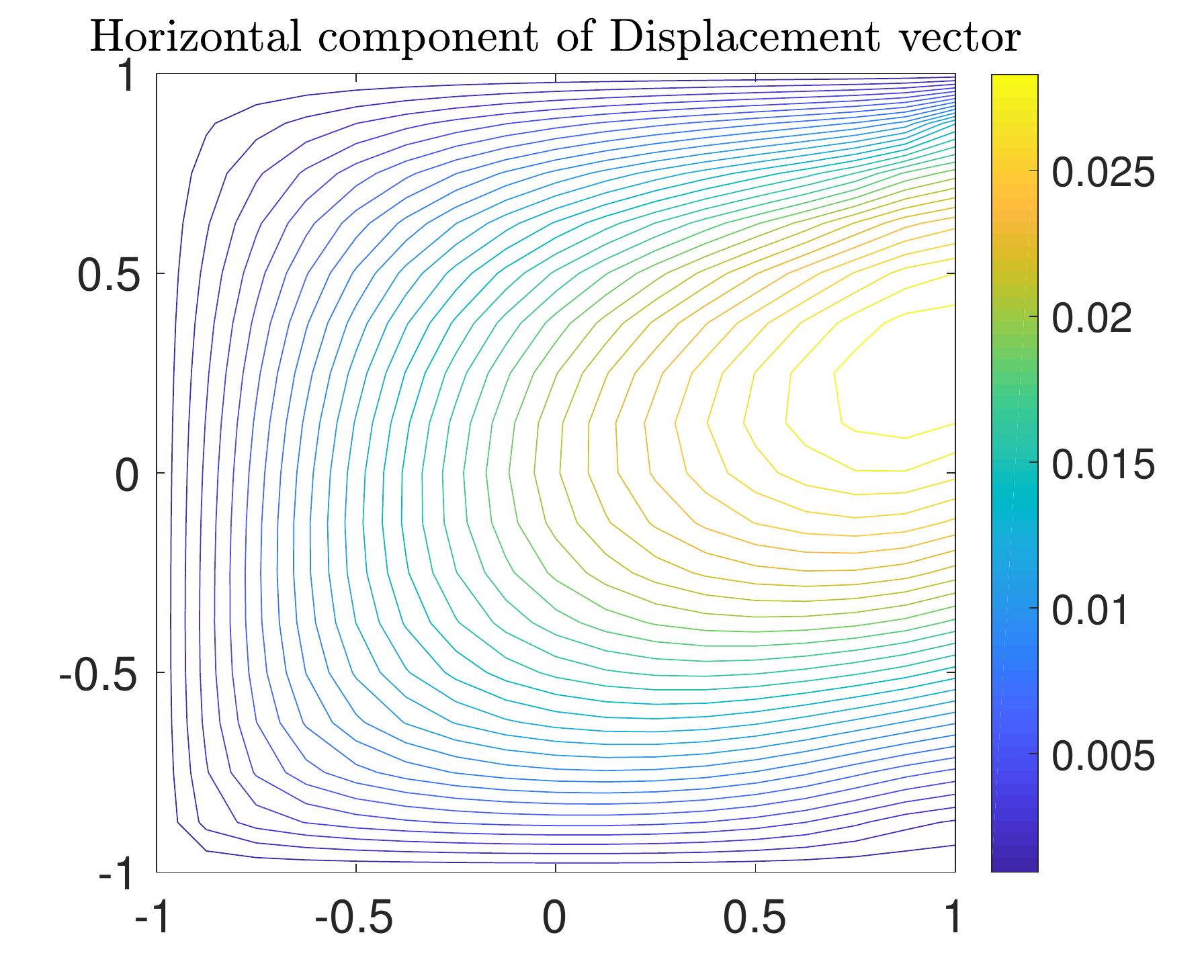}
\label{exc2bnu4DEFconf}
\centering
\includegraphics[width=.43\textwidth]{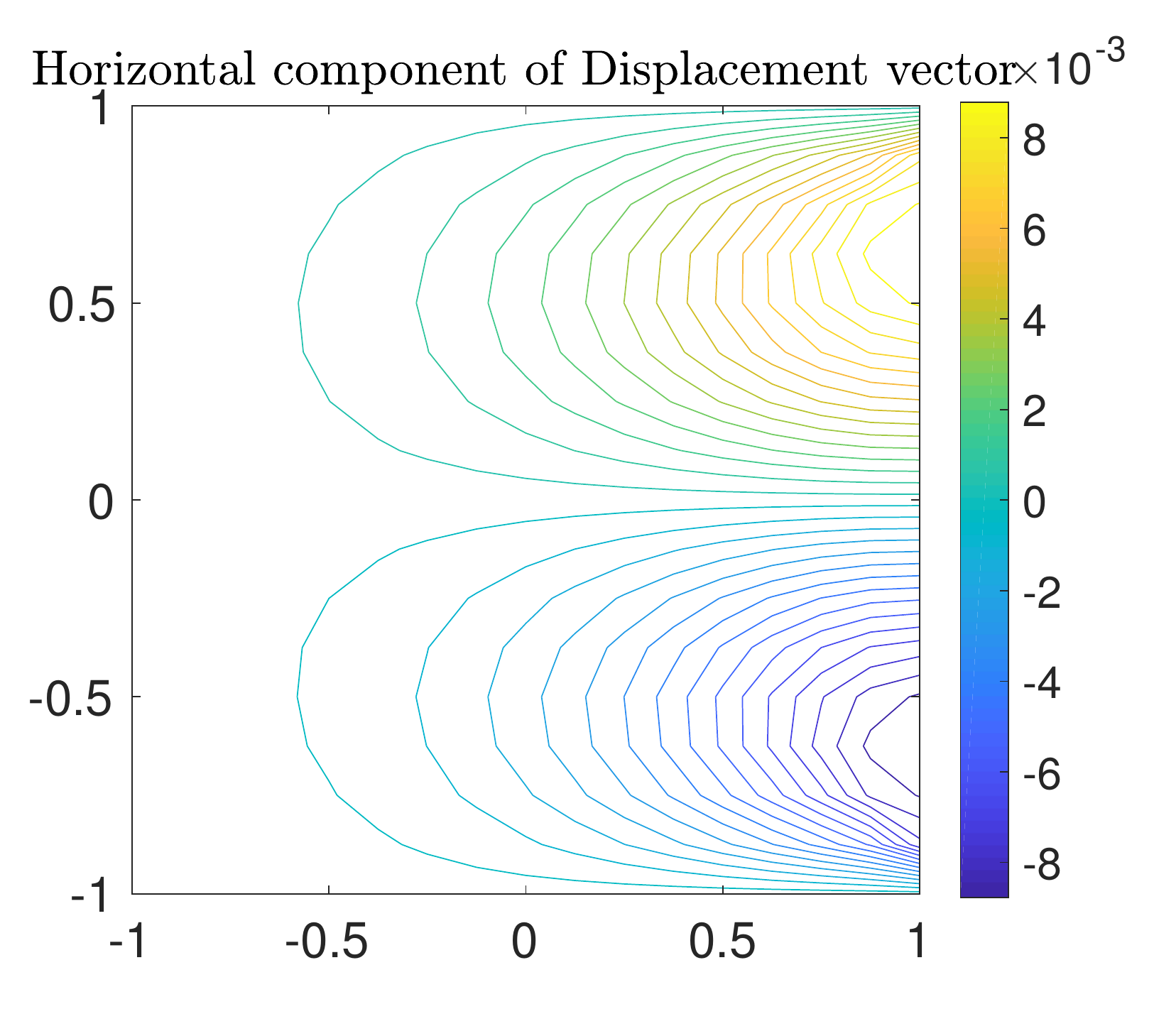}
\label{exc2bnu49999DEFconf}
\caption
{Contour plot of first component of displacement vector for $h=1/4$, with $\mu=10$ and $\nu=.4$ (left); $\nu=0.49999$ (right).}
\label{Ex1c2bDEFconf}
\end{figure}
\begin{figure}
\centering
\includegraphics[width=.46\textwidth]{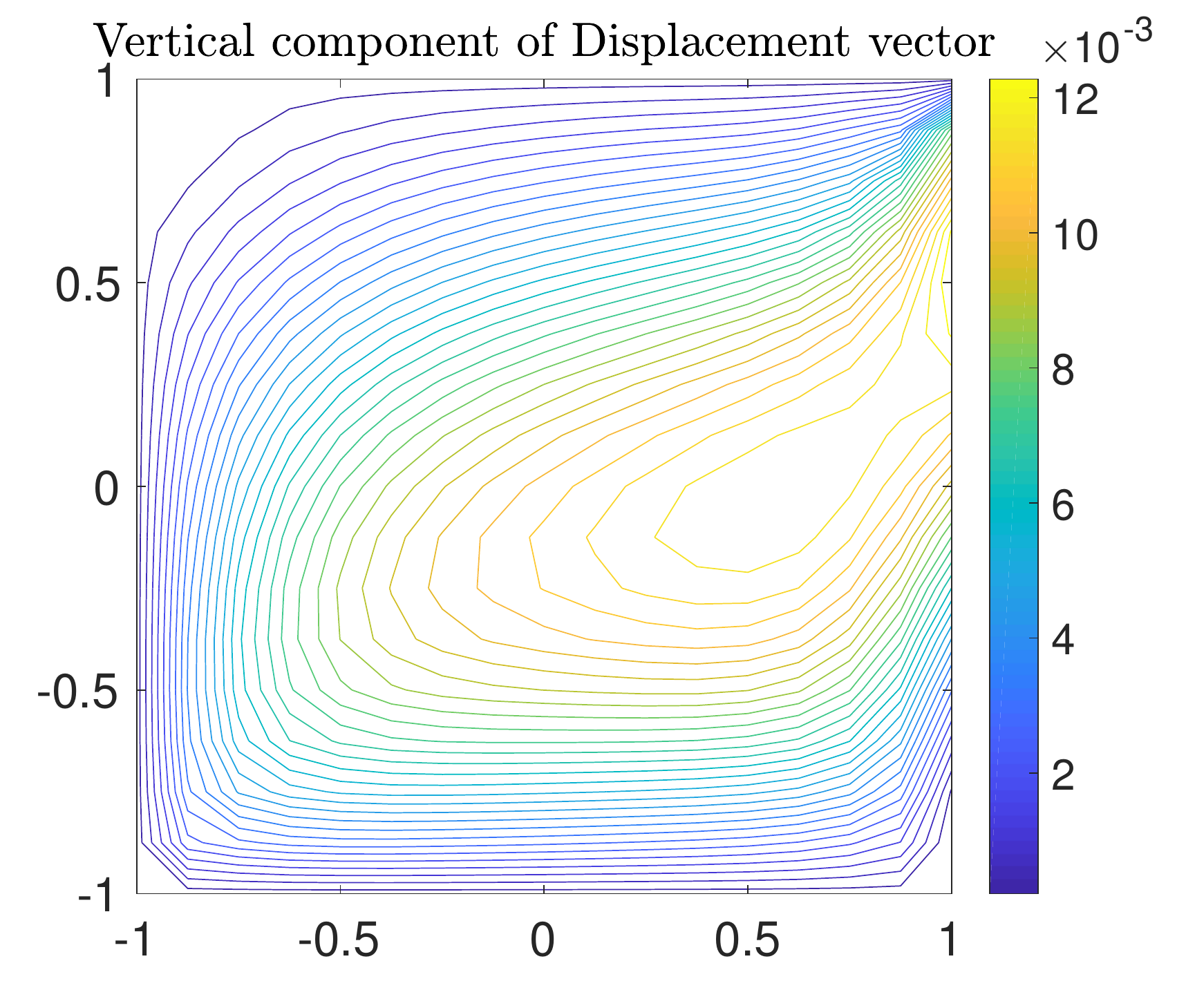}
\label{exc2bnu4DEFcons}
\centering
\includegraphics[width=.46\textwidth]{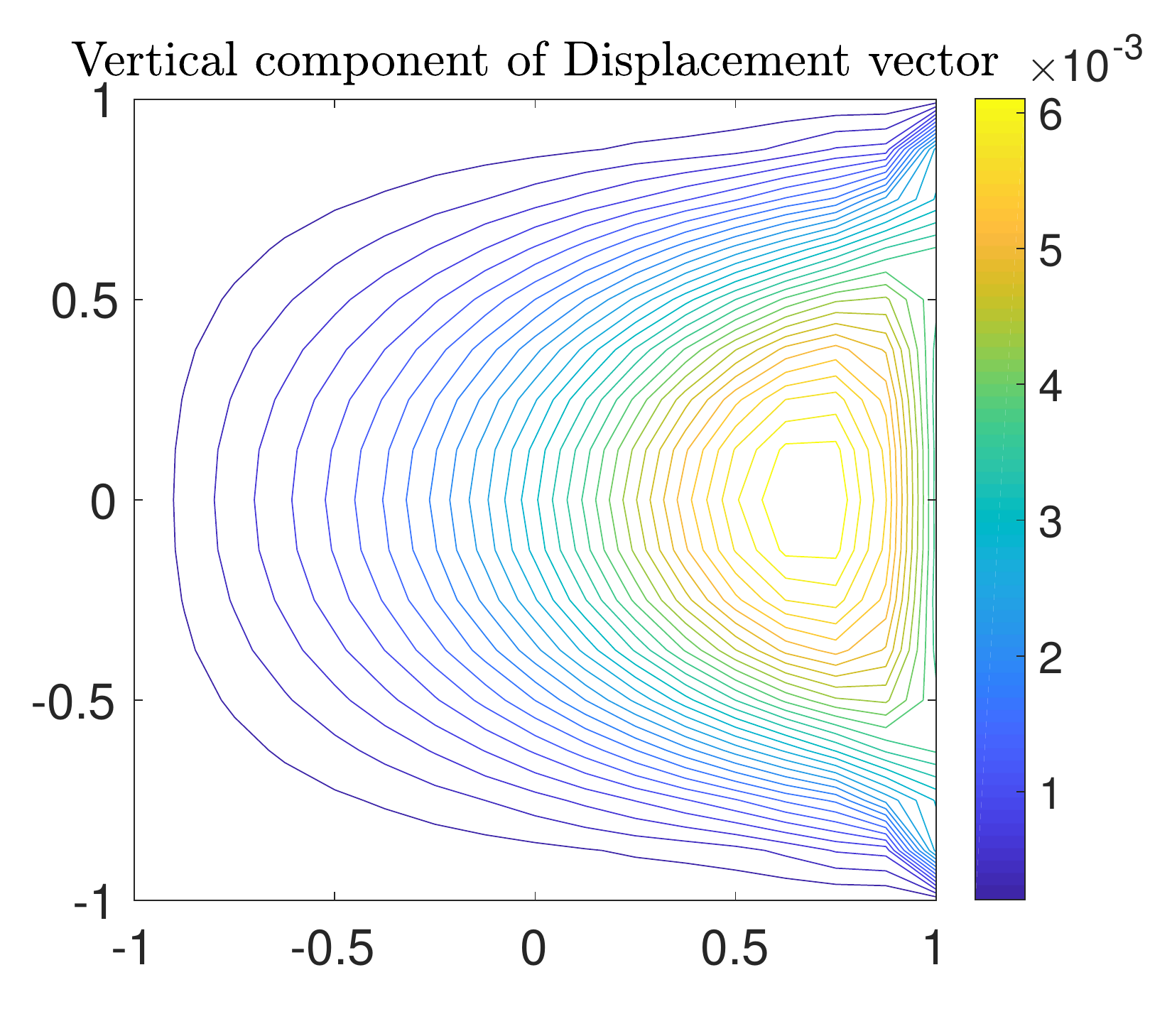}
\label{exc2bnu49999DEFcons}
\caption
{Contour plot of second component of displacement vector for $h=1/4$, with $\mu=10$ and $\nu=.4$ (left); $\nu=0.49999$ (right).}
\label{Ex1c2bDEFcons}
\end{figure}

\begin{figure}
\centering
\includegraphics[width=.40\textwidth]{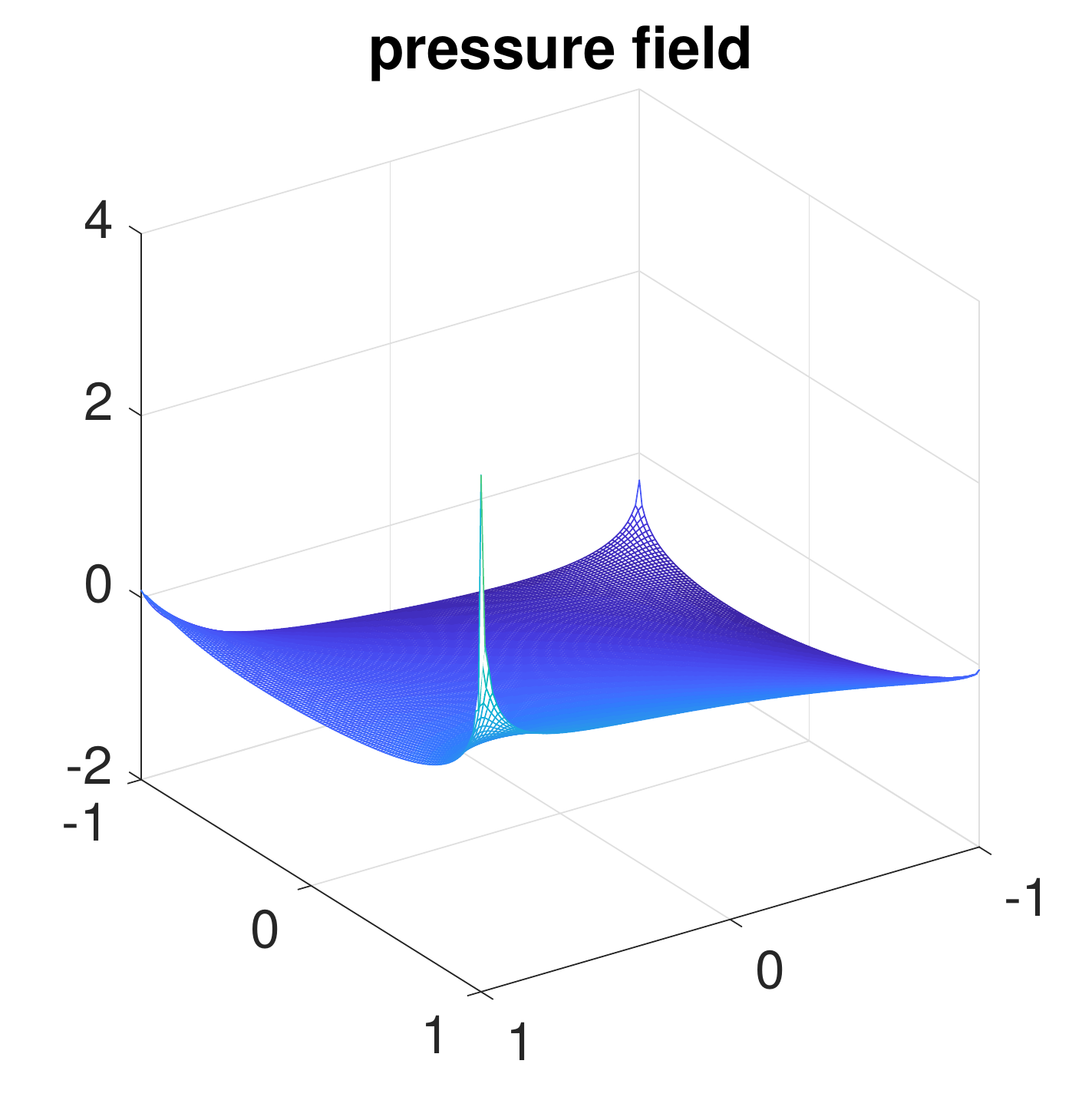}
\label{exc2bnu4PRE}
\hspace{0.005cm}
\centering
\includegraphics[width=.40\textwidth]{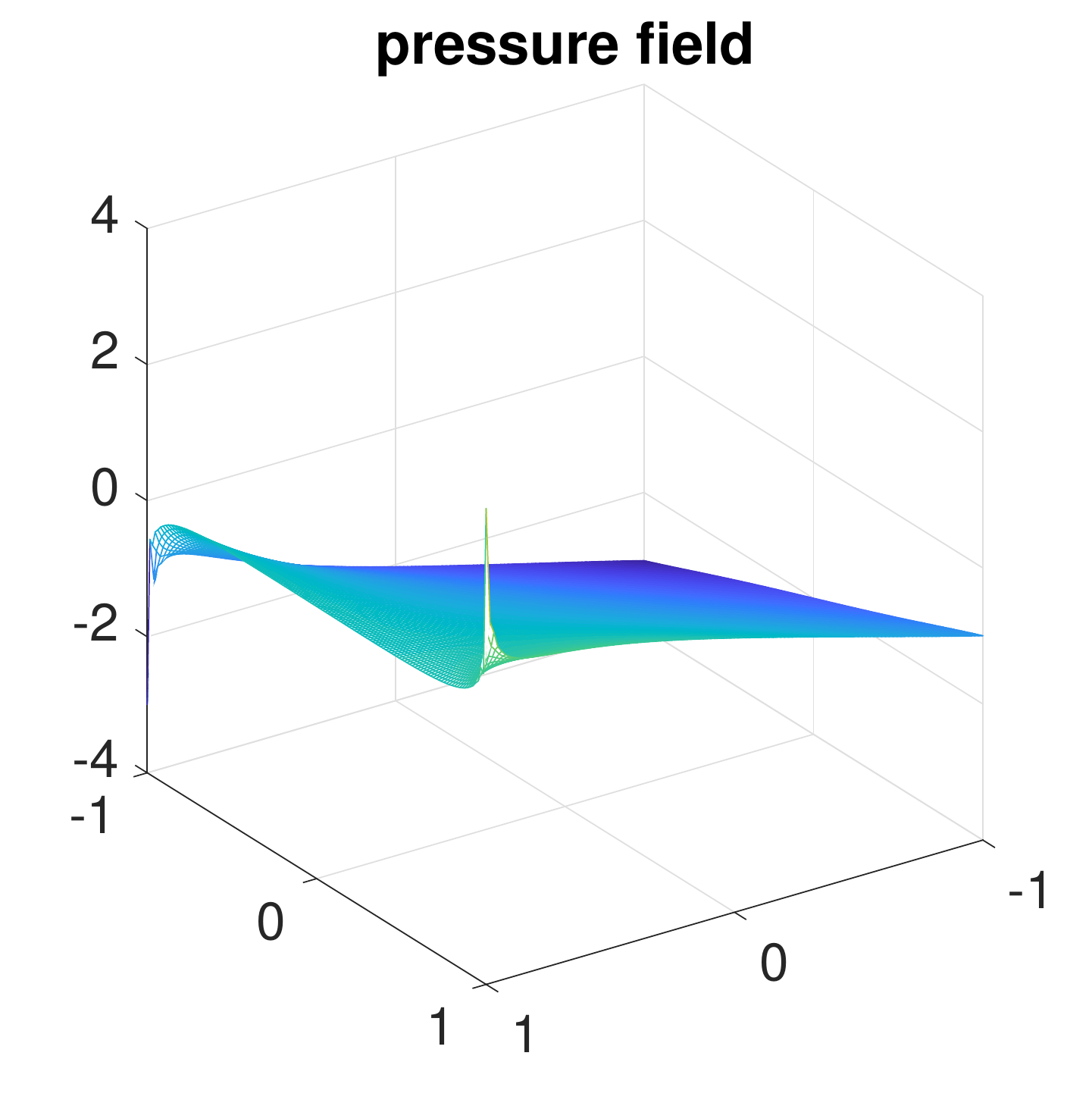}
\label{exc2bnu49999PRE}
\caption
{Computed pressure field for  $h=1/64$, $\mu=10$ and  $\nu=.4$ (left);   $\nu=0.49999$  (right).}
\label{Ex1c2bPRE}
\end{figure}

The solution to this problem does not even have $H^2$ regularity: there is a strong singularity at the top right corner, where the boundary condition changes from essential to natural, and weaker singularities at the points $(-1,1)$ and { $(-1,-1)$ for $\nu=0.4$}. {In other case $\nu=0.49999$, there are two strong singularities at two corners, where the boundary condition changes from essential to natural.} The computed deformations of the elastic body for two representative values of the Poisson ratio $\nu$  are shown in Figure~\ref{Ex1c2bDEF} and the associated (rotated) pressure solutions are presented in Figure~\ref{Ex1c2bPRE}. The strength  of the singularity at the corner $(1,1)$ is very evident in the computed pressure field {for $\nu=0.4$}. {But, for $\nu=0.49999$, the two singularities at the corners $(1,1)$ and $(-1,1)$ are very evident and have opposite directions in the computed pressure field.}
  A plot of the element contributions to the Poisson estimator $\eta_P$  that is computed on the same grid is shown in Figure~\ref{Ex1c2b49999Est}. We see that the error estimator does a good job in identifying the position and relative strength of  the 
  \rbl{underlying} singularities. {In Figure~\ref{errex3Est},} the lack of smoothness in the solution is reflected in the  convergence rate (slower than  ${\mathcal O}(h)$) of all three error estimates  when the grid is refined  uniformly. 

\begin{figure}[!th]
\centering
\includegraphics[width=.4\textwidth]{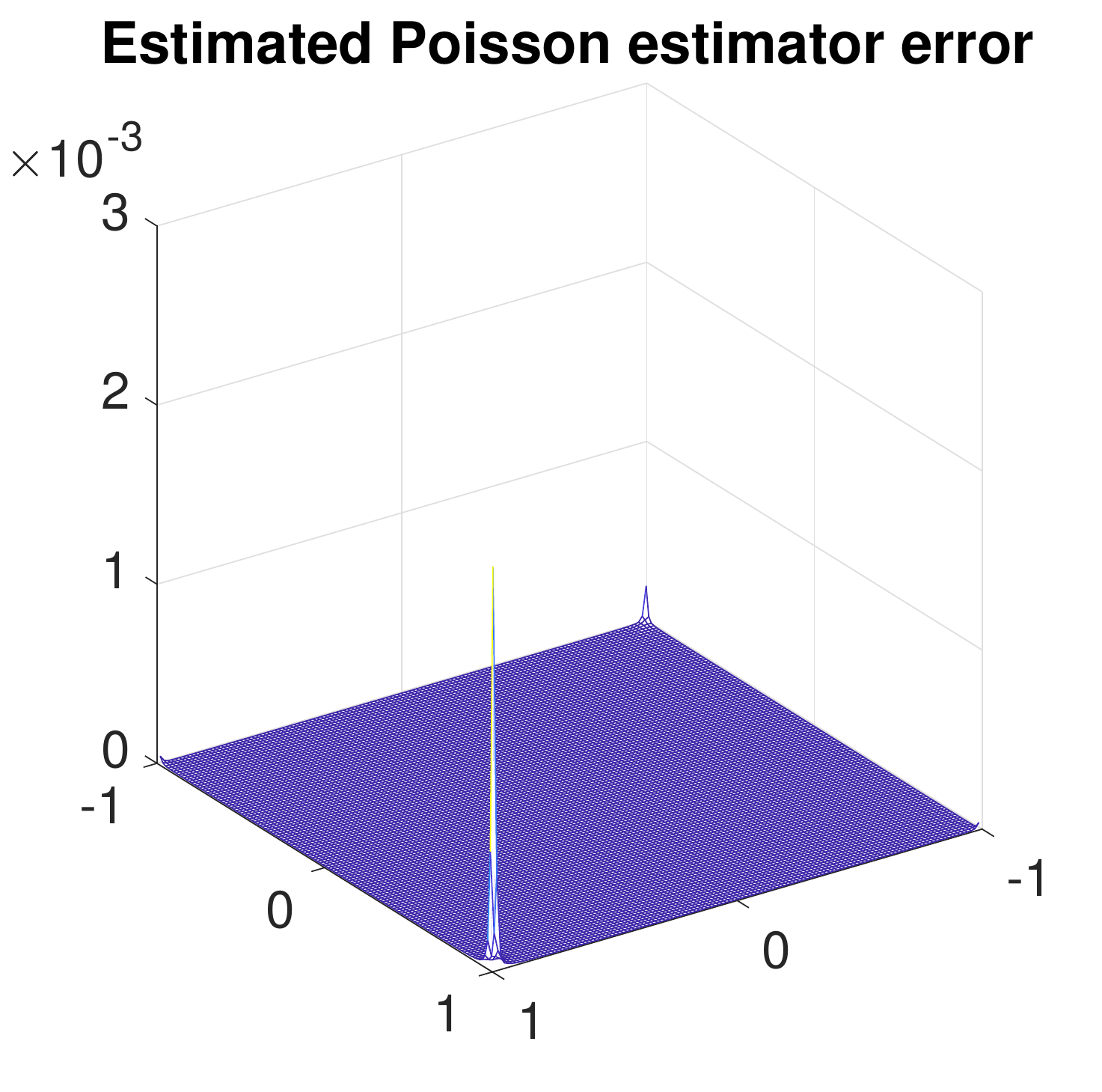}
\label{exc2bnu49999PEE}
\centering
\includegraphics[width=.4\textwidth]{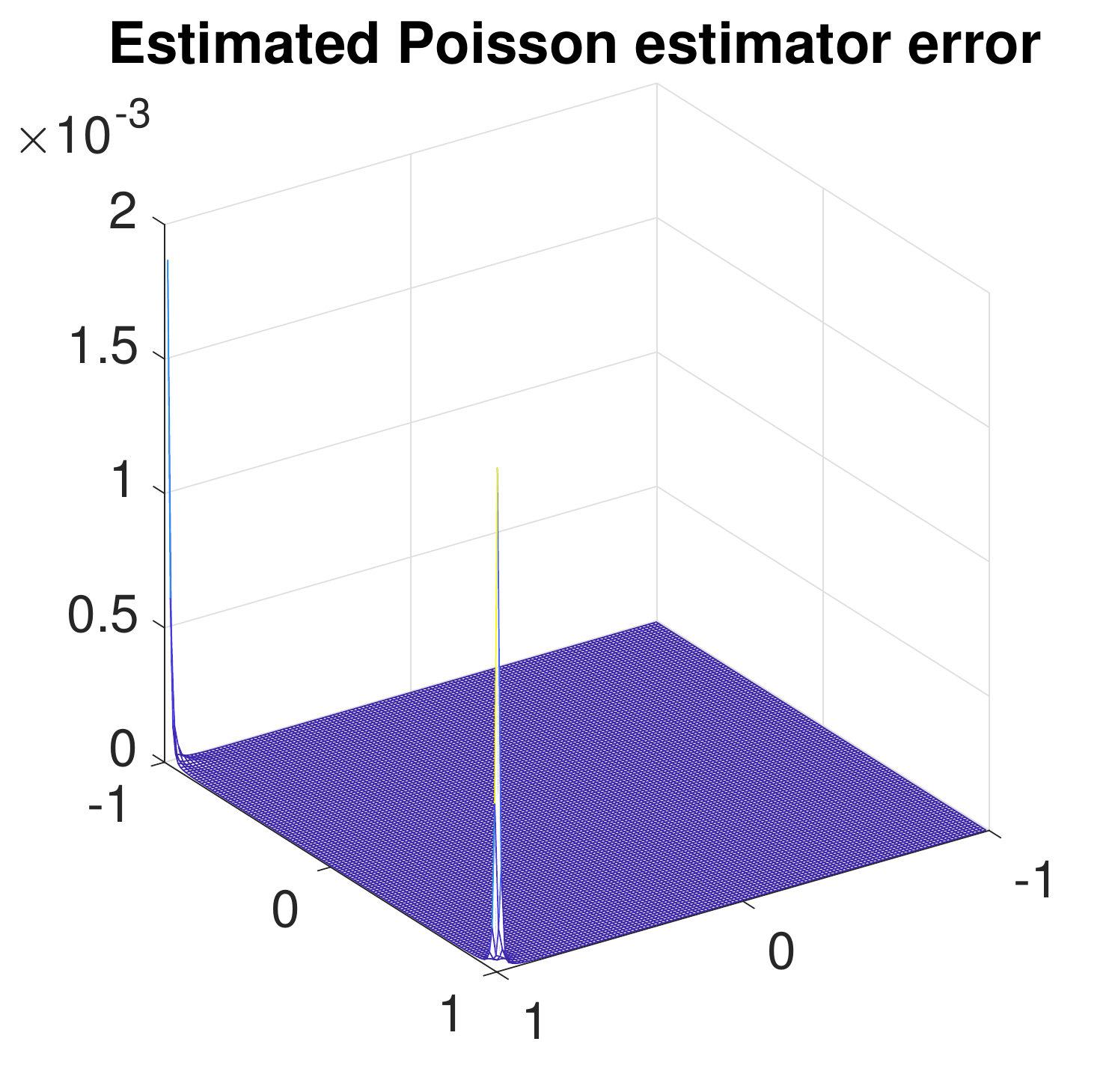}
\label{exc2bnu49999PEE1}
\caption
{Mesh plot of the element contributions $\eta_{P,K}$ to the local Poisson estimator $\eta_{P}$ for the third test problem computed with $h=1/64$, for the case $\mu=10$ and $\nu=0.4$(left), $\nu=0.49999$ (right).}
\label{Ex1c2b49999Est}
\end{figure}
\begin{figure}[!th]
\centering
\includegraphics[width=.4\textwidth]{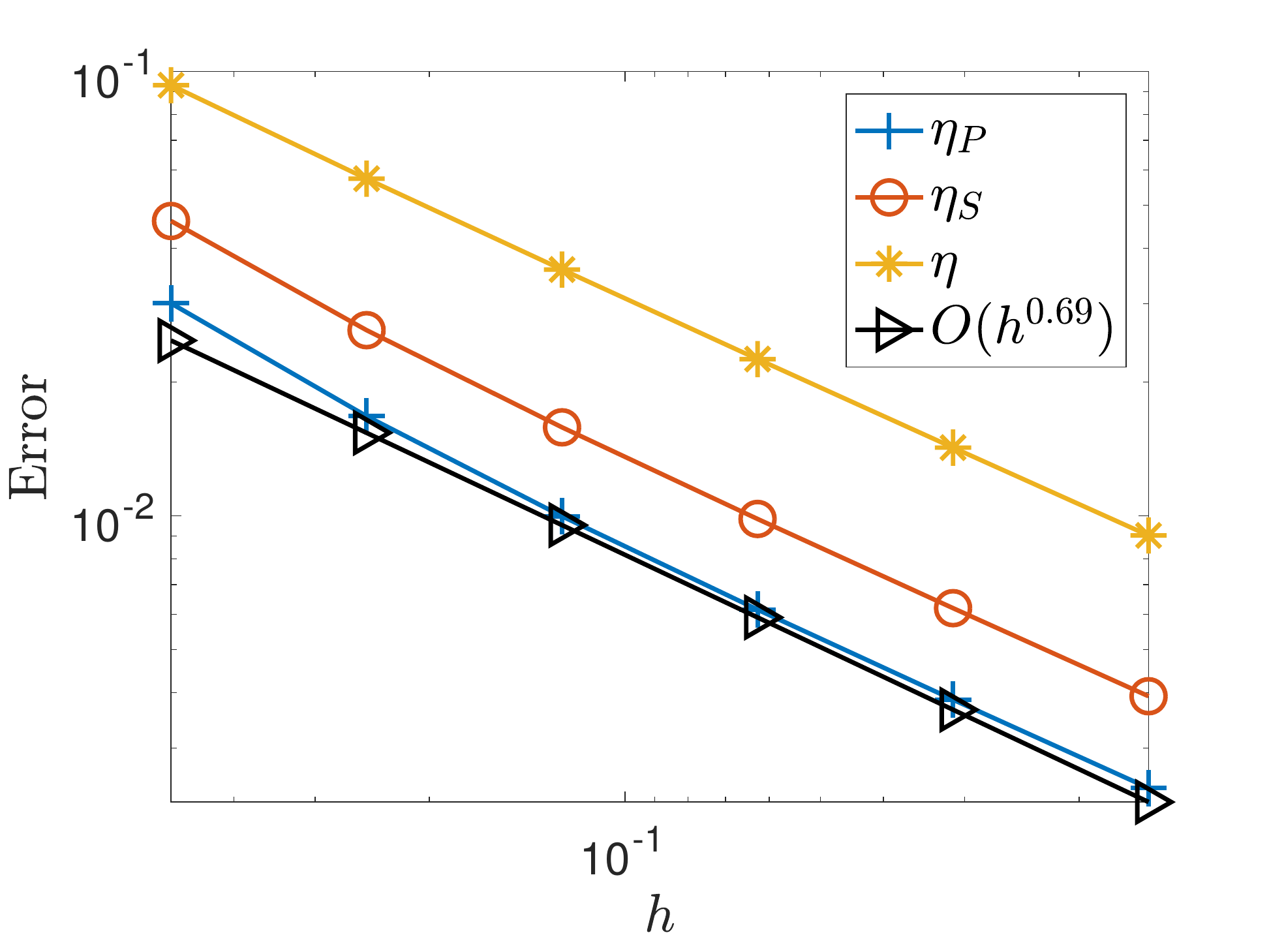}
\label{errex3Est04}
\centering
\includegraphics[width=.4\textwidth]{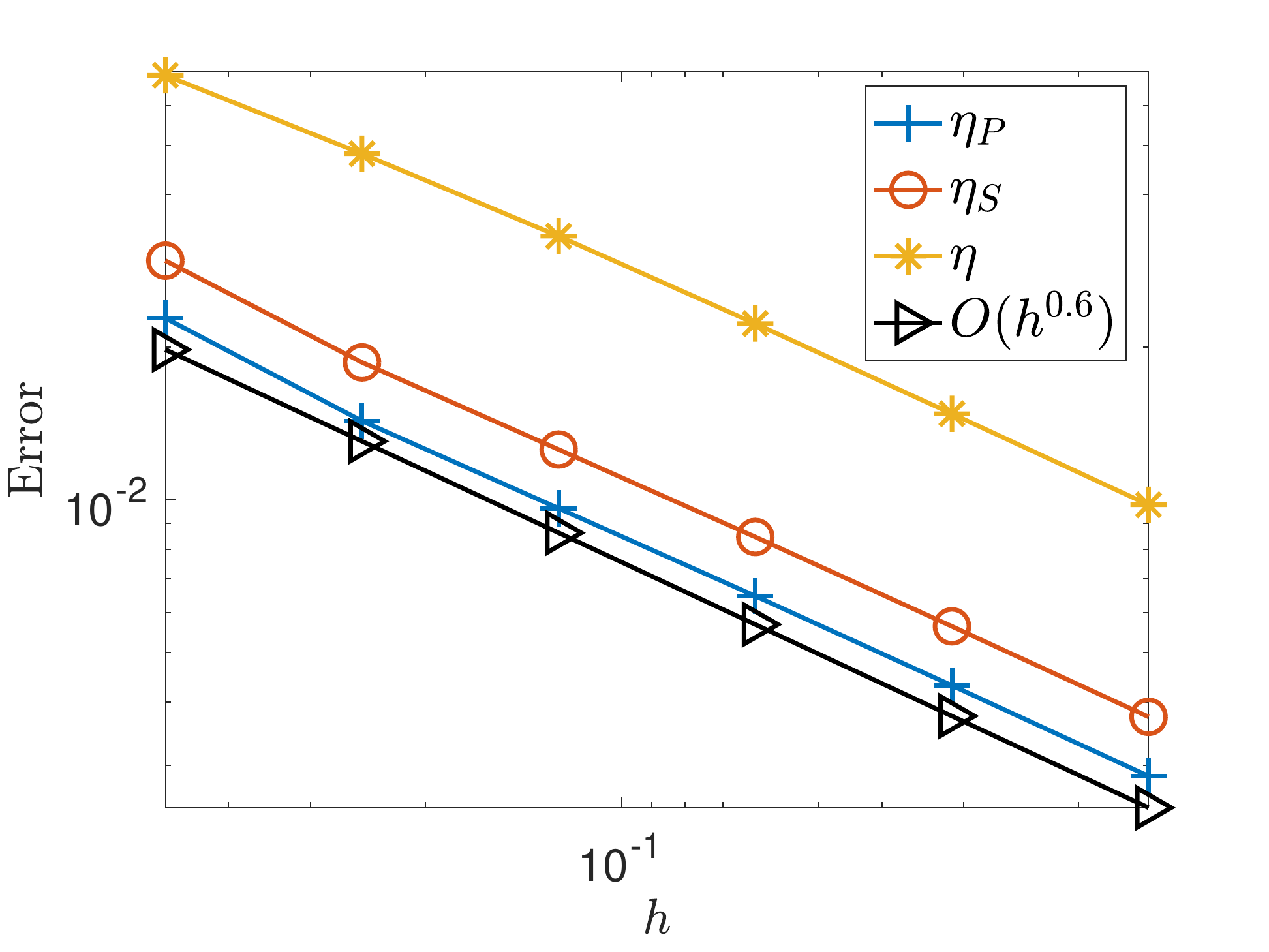}
\label{errex3Est049999}
\caption
{Estimated energy errors computed with the residual-based ($\eta$), local Stokes problem $(\eta_{S})$ and local Poisson problem ($\eta_{P}$) estimators, for varying mesh size $h$ and parameters $\mu=10$ and $\nu=.4$ (left); $\mu=10$  and $\nu=0.49999$  (right).}
\label{errex3Est}
\end{figure}
\section{Concluding remarks}\label{conclusions}

There are two important contributions in this paper. First, we have  developed some new {robust} error estimators for computing locking-free approximations of linear elasticity problems.  We have shown that these estimators give reliable estimates of the approximation error even when working arbitrarily close to the incompressible limit $\nu=1/2$. Second, we have identified a practical error estimation strategy based on solving uncoupled Poisson problems for the displacement components that yields effectivity indices close to unity in all the cases tested. Extending this work to enable the adaptive solution of elasticity problems with {\it uncertain} material parameters is the subject of ongoing research. Ensuring robustness in the error estimation process is fundamentally important when solving problems with large variability in the measurement of such parameters.


\bibliographystyle{siam}
\bibliography{kps}

\begin{thebibliography}{10}

\bibitem{MJ}
{\sc Mark Ainsworth and J.~Tinsley Oden}, {\em A Posteriori Error Estimation in
  Finite Element Analysis}, Wiley, 2000.

\bibitem{arnold2007mixed}
{\sc Douglas Arnold, Richard Falk, and Ragnar Winther}, {\em Mixed finite
  element methods for linear elasticity with weakly imposed symmetry},
  Mathematics of Computation, 76 (2007), pp.~1699--1723.

\bibitem{arnold2002mixed}
{\sc Douglas~N Arnold and Ragnar Winther}, {\em Mixed finite elements for
  elasticity}, Numerische Mathematik, 92 (2002), pp.~401--419.

\bibitem{barrios2006residual}
{\sc Tom{\'a}s~P Barrios, Gabriel~N Gatica, Mar{\'\i}a Gonz{\'a}lez, and
  Norbert Heuer}, {\em A residual based a posteriori error estimator for an
  augmented mixed finite element method in linear elasticity}, ESAIM:
  Mathematical Modelling and Numerical Analysis, 40 (2006), pp.~843--869.

\bibitem{becker2009nitsche}
{\sc Roland Becker, Erik Burman, and Peter Hansbo}, {\em A nitsche extended
  finite element method for incompressible elasticity with discontinuous
  modulus of elasticity}, Computer Methods in Applied Mechanics and
  Engineering, 198 (2009), pp.~3352--3360.

\bibitem{DFM}
{\sc Daniele Boffi, Franco Brezzi, and Michel Fortin}, {\em Mixed Finite
  Element Methods and Applications}, Springer, Heidelberg, 2013.
\newblock {\tt http://dx.doi.org/10.1007/978-3-642-36519-5}.

\bibitem{DR}
{\sc Daniele Boffi and Rolf Stenberg}, {\em A remark on finite element schemes
  for nearly incompressible elasticity}, Computers and Mathematics with
  Applications, 74 (2017), pp.~2047--2055.
\newblock {\tt http://dx.doi.org/10.1016/j.camwa.2017.06.006}.

\bibitem{SB}
{\sc Susanne~C. Brenner}, {\em {K}orn's inequalities for piecewise {H}$^{1}$
  vector fields}, Math. Comp., 73 (2003), pp.~1067--1087.

\bibitem{SCL}
{\sc Susanne~C. Brenner and Li-Yeng Sung}, {\em Linear finite element methods
  for planar linear elasticity}, Math. Comp., 59 (1992), pp.~321--338.

\bibitem{CJ}
{\sc Carsten Carstensen and Joscha Gedicke}, {\em Robust residual-based a
  posteriori {A}rnold--{W}inther mixed finite element analysis in elasticity},
  Comput. Methods Appl. Mech. Engrg, 300 (2016), pp.~245--264.

\bibitem{CLA}
{\sc Philippe Cl\'{e}ment}, {\em Approximation by finite element functions
  using local regularization}, R.A.I.R.O. Anal. Num\'{e}r., 2 (1975),
  pp.~77--84.

\bibitem{ifiss}
{\sc Howard Elman, Alison Ramage, and David Silvester}, {\em {IFISS}: A
  computational laboratory for investigating incompressible flow problems},
  SIAM Review, 56 (2014), pp.~261--273.
\newblock {\tt http://dx.doi.org/10.1137/120891393}.

\bibitem{HDA}
{\sc Howard Elman, David Silvester, and Andy Wathen}, {\em Finite Elements and
  Fast Iterative Solvers: with Applications in Incompressible Fluid Dynamics},
  Oxford University Press, Oxford, UK, 2014.
\newblock Second Edition, xiv+400 pp. ISBN: 978-0-19-967880-8.

\bibitem{VGPAR}
{\sc Vivette Girault and Pierre-Arnaud Raviart}, {\em Finite Element Methods
  for Navier--Stokes Equations}, Springer, Berlin, 1986.

\bibitem{RLH}
{\sc Leonard~R. Herrmann}, {\em Elasticity equations for incompressible and
  nearly incompressible materials by a variational theorem}, AIAA J., 3 (1965),
  pp.~1896--1900.

\bibitem{PDT}
{\sc Paul Houston, Dominik Sch{\" o}tzau, and Thomas~P. Wihler}, {\em An
  hp-adaptive mixed discontinuous {G}alerkin {FEM} for nearly incompressible
  linear elasticity}, Comput. Methods Appl. Mech. Engrg, 195 (2006),
  pp.~3224--3246.

\bibitem{TJH}
{\sc Thomas J.~R. Hughes}, {\em The Finite Element Method}, Prentice-Hall, New
  Jersey, 1987.

\bibitem{KS}
{\sc Reijo Kouhia and Rolf Stenberg}, {\em A linear nonconforming finite
  element method for nearly incompressible elasticity and {S}tokes flow},
  Comput. Methods Appl. Mech. Engrg, 124 (1995), pp.~195--212.

\bibitem{QLDS}
{\sc Qifeng Liao and David Silvester}, {\em A simple yet effective a posteriori
  error estimator for classical mixed approximation of {S}tokes equations},
  Appl. Numer. Math., 62 (2012), pp.~1242--1256.
\newblock {\tt http://dx.doi.org/10.1016/j.apnum.2010.05.003}.

\bibitem{lonsing2004posteriori}
{\sc Marco Lonsing and R{\"u}diger Verf{\"u}rth}, {\em A posteriori error
  estimators for mixed finite element methods in linear elasticity}, Numerische
  Mathematik, 97 (2004), pp.~757--778.

\bibitem{DHA}
{\sc David Silvester, Howard Elman, and Alison Ramage}, {\em Incompressible
  flow and iterative solver software ({IFISS}), version 3.5}, September 2016.
\newblock {\tt http://www.manchester.ac.uk/ifiss/}.

\bibitem{RV}
{\sc Rudiger Verf{\" u}rth}, {\em A Posteriori Error Estimation Techniques for
  Finite Element Methods}, Oxford University Press, Oxford, 2013.

\end{thebibliography}

\end{document}